\documentclass[a4paper,USenglish%,cleveref,autoref,thm-restate
]{article}

\usepackage{amsmath,amssymb,mathtools,amsthm}
\usepackage{graphics}
\usepackage{thmtools} 
\usepackage{thm-restate}

\usepackage{tikz}
\usetikzlibrary{intersections}

 \newtheorem{theorem}{Theorem}
 
 \newtheorem{proposition}[theorem]{Proposition}
 \newtheorem{lemma}[theorem]{Lemma}
 \newtheorem{definition}[theorem]{Definition}
 \newtheorem{corollary}[theorem]{Corollary}
 \newtheorem{claim}[theorem]{Claim}

 \usepackage[
 colorinlistoftodos,prependcaption]{todonotes}
 \usepackage{setspace}
 \newcounter{mycomment}

\newcommand{\R}{\mathbb R}
\newcommand{\E}{\mathbb E}
\newcommand{\N}{\mathbb N}
\newcommand{\Z}{\mathbb Z}
\newcommand{\T}{\mathsf T}
\newcommand{\sfe}{\mathbb{S}^{n-1}}

\renewcommand{\subset}{\subseteq}

\newcommand{\argmax}{\operatorname{argmax}}
\newcommand{\argmin}{\operatorname{argmin}}

\newcommand{\eps}{\varepsilon}
 \newcommand{\conv}{\operatorname{conv}}
\newcommand{\verts}{\operatorname{vertices}}

\newcommand{\diam}{\operatorname{diam}}
\newcommand{\Var}{\operatorname{Var}}
\newcommand{\Pois}{\operatorname{Pois}}
\newcommand{\ball}{\mathbb{B}_2}
\DeclarePairedDelimiter\abs{\lvert}{\rvert}%
\DeclarePairedDelimiter\norm{\lVert}{\rVert}%
\newcommand{\sprod}[2]{\langle #1, #2 \rangle}

\usepackage{geometry}

\usepackage{hyperref,cleveref}
\hypersetup{
	colorlinks=true,
	linkcolor=purple, %blue
	citecolor=purple,   
	urlcolor=black,
}

\title{Asymptotic Bounds on the Combinatorial Diameter of Random Polytopes}

\author{Gilles Bonnet%
\footnote{\href{mailto:g.f.y.bonnet@rug.nl}{g.f.y.bonnet@rug.nl};
	Bernoulli Institute for Mathematics, Computer Science and Artificial
	Intelligence, University of Groningen, The Netherlands;
	CogniGron (Groningen Cognitive Systems and Materials Center), University of Groningen, The Netherlands.} 
\footnote{Partially funded by the CogniGron research center	and the Ubbo Emmius Funds (Univ. of Groningen).}
\footnote{Partially funded by the DFG Priority Program (SPP) 2265 \href{https://spp2265.wias-berlin.de/index.php}{Random Geometric Systems}, project P23.}
\\ {\small University of Groningen}
\and
Daniel Dadush%
\footnote{\href{mailto:dadush@cwi.nl}{dadush@cwi.nl};
	Centrum Wiskunde \& Informatica, The Netherlands} 
\footnote{Supported by the ERC Starting grant QIP--805241.}
\\  {\small Centrum Wiskunde \& Informatica}
\and
Uri Grupel%
\footnote{\href{mailto:uri.grupel@uibk.ac.at}{uri.grupel@uibk.ac.at};
	University of Innsbruck, Austria} 
\\ {\small University of Innsbruck}
\and
Sophie Huiberts%
\footnote{\href{mailto:s.huiberts@cwi.nl}{s.huiberts@cwi.nl};
	Centrum Wiskunde \& Informatica, The Netherlands}
\\ {\small Centrum Wiskunde \& Informatica}
\and
Galyna Livshyts%
\footnote{\href{mailto:glivshyts6@math.gatech.edu}{glivshyts6@math.gatech.edu};
	Georgia Institute of Technology, United States} 
\\ {\small Georgia Institute of Technology}
}

\begin{document}

\maketitle

\begin{abstract}
The combinatorial diameter $\diam(P)$ of a polytope $P$ is the maximum
shortest path distance between any pair of vertices. In this paper, we
provide upper and lower bounds on the combinatorial diameter of a random
``spherical'' polytope, which is tight to within one factor of dimension when
the number of inequalities is large compared to the dimension. More
precisely, for an $n$-dimensional polytope $P$ defined by the intersection of
$m$ i.i.d.\ half-spaces whose normals are chosen uniformly from the sphere,
we show that $\diam(P)$ is $\Omega(n m^{\frac{1}{n-1}})$ and $O(n^2
m^{\frac{1}{n-1}} + n^5 4^n)$ with high probability when $m \geq
2^{\Omega(n)}$. 

For the upper bound, we first prove that the number of vertices in any fixed
two dimensional projection sharply concentrates around its expectation when
$m$ is large, where we rely on the $\Theta(n^2 m^{\frac{1}{n-1}})$ bound on
the expectation due to Borgwardt [Math. Oper. Res., 1999]. To obtain the
diameter upper bound, we stitch these ``shadows paths'' together over a
suitable net using worst-case diameter bounds to connect vertices to the
nearest shadow. For the lower bound, we first reduce to lower bounding the
diameter of the dual polytope $P^\circ$, corresponding to a random convex
hull, by showing the relation $\diam(P) \geq (n-1)(\diam(P^\circ)-2)$. We
then prove that the shortest path between any ``nearly'' antipodal pair
vertices of $P^\circ$ has length $\Omega(m^{\frac{1}{n-1}})$.

\end{abstract}

\newpage

\tableofcontents
\section{Introduction}

When does a polyhedron have small (combinatorial) diameter? This question has
fascinated mathematicians, operation researchers and computer scientists for
more than half a century. In a letter to Dantzig in 1957, motivated by the
study of the simplex method for linear programming, Hirsch conjectured that
any $n$-dimensional polytope with $m$ facets has diameter at most $m-n$.
While recently disproved by Santos~\cite{santos} (for unbounded polyhedra,
counter-examples were already given by Klee and Walkup~\cite{klee1967d}), the
question of whether the diameter is bounded from above by a polynomial in $n$ and $m$,
known as the \emph{polynomial Hirsch conjecture}, remains wide open. In fact,
the current counter-examples violate the conjectured $m-n$ bound by at most 25
percent.  

% A popular weakening of the Hirsch conjecture is to ask whether the diameter
% is bounded by a polynomial in $m$ and $n$, known as the \emph{polynomial
% Hirsch conjecture}, which remains wide-open. 

The best known general upper bounds on the combinatorial diameter of
polyhedra are the $2^{n-3} m$ bound by Barnette and Larman \cite{barnette1,
larman, barnette2}, which is exponential in $n$ and linear in $m$, and the
\emph{quasi-polynomial} $m^{\log_2 n + 1}$ bound by Kalai and Kleitman
\cite{kalaikleitman}. The Kalai-Kleitman bound was recently improved to
$(m-n)^{\log_2 n}$ by Todd~\cite{todd2014improved} and $(m-n)^{\log_2
O(n/\log n)}$ by Sukegawa~\cite{sukegawa2019asymptotically}. Similar diameter
bounds have been established for graphs induced by certain classes of
simplicial complexes, which vastly generalize $1$-skeleta of polyhedra. In
particular, Eisenbrand et al~\cite{EHRR10} proved both Barnette-Larman and
Kalai-Kleitman bounds for so-called connected-layer families (see
\Cref{thm:BL}), and Labb{\'e} et al~\cite{labbe2017hirsch} extended the
Barnette-Larman bound to pure, normal, pseudo-manifolds without boundary.

Moving beyond the worst-case bounds, one may ask for which families of
polyhedra does the Hirsch conjecture hold, or more optimistically, are there
families for which we can significantly beat the Hirsch conjecture? In the
first line, many interesting classes induced by combinatorial optimization
problems are known, including the class of polytopes with vertices in
$\{0,1\}^n$~\cite{naddef1989hirsch}, Leontief substitution
systems~\cite{grinold1971hirsch}, transportation polyhedra and their
duals~\cite{balinski1984hirsch,brightwell2006linear,borgwardt2018diameters},
as well as the fractional stable-set and perfect matching
polytopes~\cite{michini2014hirsch,sanita2018diameter}.

Related to the second vein, there has been progress on obtaining diameter
bounds for classes of ``well-conditioned'' polyhedra. If $P$
is a polytope defined by an integral constraint matrix $A \in \Z^{m \times
n}$ with all square submatrices having determinant of absolute value at most
$\Delta$, then diameter bounds polynomial in $m,n$ and $\Delta$ have been
obtained~\cite{dyer1994random,bonifas2014sub,dadush-hahnle,narayanan2021spectral}.
The best current bound is $O(n^3 \Delta^2\log(\Delta))$, due
to~\cite{dadush-hahnle}. Extending on the result of
Naddef~\cite{naddef1989hirsch}, strong diameter bounds have been proved for
polytopes with vertices in
$\{0,1,\dots,k\}^n$~\cite{kleinschmidt1992diameter,del2016diameter,deza2018improved}.
In particular,~\cite{kleinschmidt1992diameter} proved that the diameter is at
most $nk$, which was improved to $nk-\lceil n/2 \rceil$ for $k \geq
2$~\cite{del2016diameter} and to $nk-\lceil 2n/3 \rceil - (k-2)$ for $k \geq
4$~\cite{deza2018improved}.

\subsection{Diameter of Random Polytopes} With a view of beating the
Hirsch bound, the main focus on this paper will be to analyze the diameter of
random polytopes, which one may think of as well-conditioned on ``average''.
Coming both from the average case and smoothed analysis
literature~\cite{Borgwardt87,Borgwardt99,jour/jacm/ST04,jour/siamjc/Vershynin09,DH19}, there is
tantilizing evidence that important classes of random polytopes may have very
small diameters. 

In the average-case context, Borgwardt~\cite{Borgwardt87,Borgwardt99} proved
that for $P := Ax \leq 1$, $A \in \R^{m \times n}$ where the rows of $A$ are
drawn from any rotational symmetric distribution (RSD), that the expected
number of edges in any fixed 2 dimensional projection of $P$ -- the so-called
\emph{shadow bound} -- is $O(n^2 m^{\frac{1}{n-1}})$. Borgwardt also showed
that this bound is tight up to constant factors when the rows of $A$ are drawn uniformly from the
sphere, that is, the expected shadow size is $\Theta(n^2 m^{\frac{1}{n-1}})$.
In the smoothed analysis context, $A$ has the form $\bar{A} + \sigma G$,
where $\bar{A}$ is a fixed matrix with rows of $\ell_2$ norm at most $1$ and
$G$ has i.i.d.~$\mathcal{N}(0,1)$ entries and $\sigma > 0$. Bounds on the
expected size of the shadow in this context were first studied by Spielman
and Teng~\cite{jour/jacm/ST04}, later improved by~\cite{jour/siamjc/Vershynin09,DH19},
where the best current bound is $O(n^2 \sqrt{\log m}/\sigma^2)$ due
to~\cite{DH19} when $\sigma \leq \frac{1}{\sqrt{n \log m}}$. 

From the perspective of short paths, these results imply that if
one samples objectives $v,w$ uniformly from the sphere, then there is a path
between the vertices maximizing $v$ and $w$ in $P$ of expected length $O(n^2
m^{\frac{1}{n-1}})$ in the RSD model, and expected length $O(n^2 \sqrt{\log
m}/\sigma^2)$ in the smoothed model. That is, ``most pairs'' of vertices
(with respect to the distribution in the last sentence), are linked by short
expected length path. Note that both of these bounds scale either sublinearly
or logarithmically in $m$, which is far better than $m-n$. While these bounds
provide evidence, they do not directly upper bound the diameter, since this
would need to work for all pairs of vertices rather than most pairs. 
%\Comment[SH]{might be going a bit too much into the weeds, but I think the best notion of 'most pairs of vertices' being connected by short paths is obtained by randomizing only $w$ while keeping $v$ fixed.}

A natural question is thus whether the shadow bound is close to the true
diameter. In this paper, we show that this is indeed the case, in the setting
where the rows of $A$ are drawn uniformly from the sphere and when $m$ is
(exponentially) large compared to $n$. More formally, our main result is as
follows:

\begin{theorem}\label{thm:main}
Suppose that $n, m \in \mathbb{N}$ satisfy $n \geq 2$ and $m \geq
2^{\Omega(n)}$. Let $A^\T := (a_1,\dots,a_M) \in \R^{n \times M}$, where $M$
is Poisson distributed with $\E[M] = m$, and $a_1,\dots,a_M$ are sampled
independently and uniformly from $\sfe$. Then, letting $P(A) := \{x \in \R^n: Ax
\leq 1\}$, with probability at least $1 - m^{-n}$, we have that
\[
\Omega(n m^{\frac{1}{n-1}}) \leq \diam(P(A)) \leq O(n^2 m^{\frac{1}{n-1}} + n^54^n).
\]
\end{theorem}

\begin{figure}[ht]
\begin{center}
\includegraphics[height=0.2\textheight]{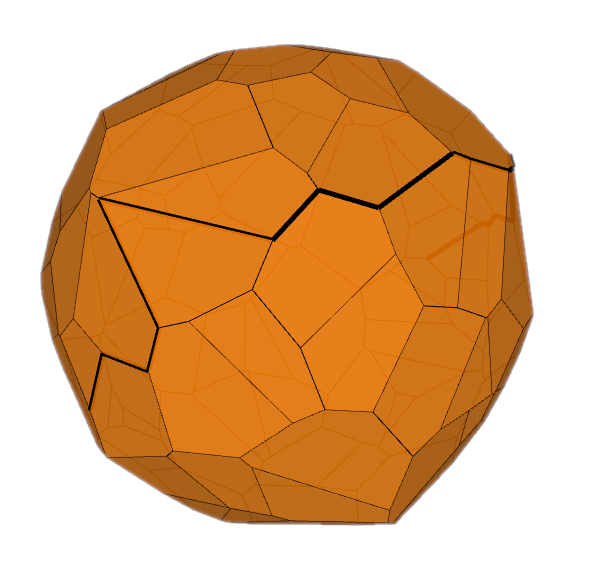}
\end{center}
\caption{Diameter Achieving Path for Random Spherical Polytope with 100
Constraints}
\end{figure}

In the above, we note that the number of constraints $M$ is chosen according
to a Poisson distribution with expectation $m$. This is only for technical
convenience (it ensures useful independence properties, see
\Cref{prop:pois-ind}), and with small modifications, our arguments also work
in the case where $M := m$ deterministically. Also, since the constraints are
chosen from the sphere, $M$ is almost surely equal to the number of facets of
$P(A)$ above (i.e., there are no redundant inequalities).  

From the bounds, we see that $\diam(P(A)) \leq O(n^2 m^{\frac{1}{n-1}})$
with high probability as
long as $m \geq 2^{\Omega(n^2)}$. This shows that the shadow bound indeed
upper bounds the expected diameter when $m \rightarrow \infty$. Furthermore, the
shadow bound is tight to within one factor of dimension in this regime. We
note that in the upper bound is already non-trivial when $m \geq \Omega(n^5 4^n)$,
since then $O(n^2 m^{\frac{1}{n-1}} + n^5 4^n) \leq m-n$.

While our bounds are only interesting when $m$ is exponential, the bounds are
nearly tight asymptotically, and as far as we are aware, they represent the
first non-trivial improvements over worst-case upper bounds for a natural
class of polytopes defined by random halfspaces. 

Our work naturally leaves two interesting open problems. The first is whether
the shadow bound upper bounds the diameter when $m$ is polynomial in $n$. The
second is to close the factor $n$ gap between upper and lower bound in the
large $m$ regime.

\subsection{Prior work} 
Lower bounds on the diameter of $P(A)$, $A^\T = (a_1,\dots,a_m) \in \R^{n
\times m}$, were studied by Borgwardt and Huhn~\cite{borgwardt_huhn_1999}.
They examined the case where each row is sampled from a RSD with radial
distribution $\Pr_a[\norm{a}_2 \leq r] = \frac{\int_0^r (1-t^2)^{\beta}
t^{n-1} dt}{\int_0^1 (1-t^2)^\beta t^{n-1} dt}$, for $r \in [0,1]$, $\beta
\in (-1,\infty)$. Restricting their results to the case $\beta \rightarrow
-1$, corresponding to the uniform distribution on the sphere (where the bound
is easier to state), they show that $\E[\diam(P(A))] \geq
\Omega(m^{\frac{1}{n}+\frac{1}{n(n-1)^2}})$. We improve their lower bound to
$\Omega(n m^{1/(n-1)})$ when $m \geq 2^{\Omega(n)}$, noting that $m^{1/(n-1)}
= O(1)$ for $m = 2^{O(n)}$. 

In terms of upper bounds, the diameter of a \emph{random convex hull of
points}, instead of a random intersection of halfspaces, has been implicitly
studied. Given $A^\T = (a_1,\dots,a_m) \in \R^{n \times m}$, let us define
\begin{equation}
\label{def:qa}
Q(A) := \conv(\{a_1,\dots,a_m\})
\end{equation}
to be the convex hull of the rows of $A$. When the rows of $A$ are sampled
uniformly from $\ball^n$, the question of when the diameter of $Q(A)$ is
exactly $1$ (i.e., every pair of distinct vertices is connected by an edge)
was studied by B{\'a}r{\'a}ny and F{\"u}redi\cite{barany1988shape}. They
proved that with probability $1-o(1)$, $\diam(Q(A)) = 1$ if $m \leq 1.125^n$
and $\diam(Q(A)) > 1$ if $m \geq 1.4^n$. 

In dimension $3$, letting $a_1,\dots,a_M \in \mathbb{S}^2$ be chosen
independently and uniformly from the $2$-sphere, where $M$ is Poisson
distributed with $\E[M] = m$, Glisse, Lazard, Michel and
Pouget~\cite{glisse2016silhouette} proved that with high probability the
maximum number of edges in any $2$-dimensional projection of $Q(A)$ is
$\Theta(\sqrt{m})$. This in particular proves that the combinatorial diameter
is at most $O(\sqrt{m})$ with high probability.

It is important to note that the geometry of $P(A)$ and $Q(A)$ are strongly
related. Indeed, as long as $m = \Omega(n)$ and the rows of $A$ are drawn
from a symmetric distribution, $P(A)$ and $Q(A)$ are polars of each other.
That is, $Q(A)^\circ = P(A)$ and $P(A)^\circ := \{x \in \R^n: \langle x, y
\rangle \leq 1, \forall y \in P(A)\} = Q(A) \footnote{Precision:
$P(A)=Q(A)^\circ$ always holds and $P(A)^\circ = Q(A)$ requires that $0\in
Q(A)$ which, as a direct consequence of Wendel's theorem \cite[Theorem
8.2.1]{schneider2008stochastic}, happens with probability $1-o(1)$ when
$m\geq cn$ for some $c>2$. In general $P(A)^\circ = \conv(A\cup\{0\})$
holds.}$.

As we will see, our proof of Theorem~\ref{thm:main} will in fact imply
similarly tight diameter bounds for $\diam(Q(A))$ as for $\diam(P(A))$,
yielding analogues and generalizations of the above results, when $A^\T =
(a_1,\dots,a_M) \in \R^{n \times M}$ and $M$ is Poisson with $\E[M] = m$.
More precisely, we will show that for $m \geq 2^{\Omega(n)}$, with high
probability
\[
\Omega(m^{\frac{1}{n-1}}) \leq \diam(Q(A)) \leq O(n m^{\frac{1}{n-1}} + n^5 4^n).
\]
In essence, for $m$ large enough, our bounds for $\diam(Q(A))$ are a factor
$\Theta(n)$ smaller than our bounds for $\diam(P(A))$. This relation will be
explained in \Cref{sec:lower-bound}.

% 
% Our main result is as follows:
% 
% Combining 
% 
% We note that our upper and lower bounds for $P(A)$ differ by $(\log n^nm)^{1+1/(n-1)}
% 2^{O(n)}$ factor, which is grows exponentially in the dimension $n$. Given
% this exponential loss, our bounds are only interesting in the regime where $m
% = 2^{c n}$ or larger, for $c \geq 1$ a sufficiently large constant. In this regime, our
% upper bound for $P(A)$, namely $O(c 2^c n^2 4^n)$, yields a
% significant improvement over both the $O(2^{(c+1)n})$ Barnette-Larman bound
% and the $2^{cn \log O(n/\log n)}$ Kalai-Kleitman bound. As far as we are
% aware, this this is the first time a non-trivial improvement has been
% obtained for the combinatorial diameter of a natural class of random
% polytopes defined by random intersections of half-spaces.
% 
% For the convex hull $Q(A)$, we note that a bound of $m-1$ on the
% combinatorial diameter is already trivial, since the $1$-skeleton has $m$
% vertices. While it is easier to get combinatorial diameter bounds for $Q(A)$
% compared to $P(A)$, we find it interesting that the combinatorial diameters
% of both $Q(A)$ and $P(A)$ behave rather similarly in the regime where $m$ is
% large.
% 
% Our $\Omega((m/\log m)^{1/(n-1)})$ lower bound for the combinatorial diameter
% of $P(A)$ is a slight sharpening of the corresponding lower bound of
% \cite{borgwardt_huhn_1999} (in the case $\beta \rightarrow -1$), which we
% show also extends to $Q(A)$. Our proof is also considerably simpler than
% that of~\cite{borgwardt_huhn_1999}.

\subsection{Proof Overview}
\label{sec:techniques}

In this section, we give the high level overview of our approach for both the
upper and lower bound in \Cref{thm:main}.

% For notational simplicity in the sequel, it will be convenient to treat $A$
% as a subset of $\sfe$ instead of a matrix. For $A \subseteq \sfe$, we will
% slightly abuse notation and let $P(A) := \{x \in \R^n: \langle x, a \rangle
% \leq 1, \forall a \in A\}$ and $Q(A) := {\rm conv}(A)$.

\subsubsection{The Upper Bound}

In this overview, we will say that an event holds with high probability if it
holds with probability $1-m^{-\Omega(n)}$. To prove the upper bound on the
diameter of $P(A)$, we proceed as follows. For simplicity, we will only describe the level high strategy for achieving a $O(n^2 m^{\frac{1}{n-1}} +
2^{O(n)})$ bound. To begin, we first show that the vertices of $P(A)$
maximizing objectives in a suitable net $N$ of the sphere $\sfe$, are all
connected to the vertex maximizing $e_1$, with a path of length $O(n^2
m^{\frac{1}{n-1}} + 2^{O(n)})$ with high probability. Second, we will show
that with high probability, for all $v \in \sfe$, there is a path between the
vertex of $P(A)$ maximizing $v$ and the corresponding maximizer of closest
objective $v' \in N$ of length at most $2^{O(n)} \log m$. Since every vertex
of $P(A)$ maximizes some objective in $\sfe$, by stitching at most $4$ paths
together, we get that the diameter of $P(A)$ is at most $O(n^2
m^{\frac{1}{n-1}} + 2^{O(n)} \log m) = O(n^2 m^{\frac{1}{n-1}} + 2^{O(n)})$
with high probability.

We only explain the strategy for the first part, as the second part follows
easily from the same techniques. The key estimate here is the sharp
$\Theta(n^2 m^{\frac{1}{n-1}})$ bound on the expected number of vertices in a
fixed two dimensional projection due to
Borgwardt~\cite{Borgwardt87,Borgwardt99}, the so-called \emph{shadow bound},
which allows one to bound the expected length of paths between vertices
maximizing any two fixed objectives (see \Cref{sec:upper-bound} for a more
detailed discussion). We first strengthen this result by proving that the
size of the shadow sharply concentrates around its expectation when $m$ is
large (\Cref{thm:shadowbound}), allowing us to apply a union bound on a
suitable net of shadows, each corresponding a two dimensional plane spanned
by $e_1$ and some element of $N$ above. To obtain such concentration, we show
that the shadow decomposes into a sum of nearly independent ``local
shadows'', corresponding to the vertices maximizing a small slice of the
objectives in the plane, allowing us to apply concentration results on
nearly-independent sums. 

\paragraph*{\bf Independence via Density} We now explain the local
independence structure in more detail. For this purpose, we examine the
smallest $\epsilon > 0$ such that rows of $A$ are $\epsilon$-dense on $\sfe$,
that is, such that every point in $\sfe$ is at distance at most $\epsilon$
from some row of $A$.  Using standard estimates on the measure of spherical
caps and the union bound, one can show with high probability that $\epsilon
:= \Theta((\log m/m)^{1/m})$ and that any spherical cap of radius $t \epsilon$
contains at most $O(t^{n-1} \log m)$ rows of $A$ for any fixed $t \geq 1$
(see \Cref{lem:precise-cap} and \Cref{cor:caps}).   
 
We derive local independence from the fact that the vertex $v$ of $P(A)$
maximizing a unit norm objective $w$ is defined by
constraints $a \in A$ which are distance at most $2\epsilon$ from $w$ (see
\Cref{lem:local-hardest-part} for a more general statement). This locality
implies that the number of vertices in a projection of $P(A)$ onto a two
dimensional subspace $W \ni w$ maximizing objectives at distance $\epsilon$
from $w$ (i.e., the slice of objectives) depends only on the constraints in
$A$ at distance at most $O(\epsilon)$ from $w$. In particular, the number of
relevant constraints for all objectives at distance $\epsilon$ from $w$ is at
most $2^{O(n)}\log m$ by the estimate in the last paragraph. By the
independence properties of Poisson processes (see \Cref{prop:pois-ind}), one
can in fact conclude that this local part of the shadow on $W$ is independent
of the constraints in $A$ at distance more than $O(\epsilon)$ from $w$. 

Given the above, we decompose the shadow onto $W$ into $k = O(1/\epsilon)$
pieces, by placing $k$ equally spaced objectives $w_0,\dots,w_{k-1},w_k =
w_0$ on $\sfe \cap W$, so that $\|w_i-w_{i+1}\|_2 \leq \epsilon$, $0 \leq i \leq k-1$,
and defining $K_i \geq 0$, $0 \leq i \leq k-1$, to be the
number of vertices maximizing objectives in $[w_i,w_{i+1}]$. This
subdivision partitions the set of shadow vertices, so Borgwardt's bound applies to the
expected sum: $\E[\sum_{i=0}^{k-1} K_i] = O(n^2 m^{1/(n-1)})$. Furthermore,
as argued above, each $K_i$ is (essentially) independent of all $K_j$'s with
$|i-j \mod k| = \Omega(1)$. This allows us to apply a Bernstein-type
concentration bound for sums of nearly-independent bounded random variables to $\sum_{i=0}^{k-1} K_i$
(see~\Cref{lem:bernstein}).

Unfortunately, the worst-case upper bounds we have for each $K_i$, $0 \leq i
\leq k-1$, are rather weak. Namely, we only know that in the the worst-case,
$K_i$ is bounded by the total number of vertices induced by constraints
relevant to the interval $[w_i,w_{i+1}]$, where $\|w_i-w_{i+1}\| \leq
\epsilon$. As mentioned above, the number of relevant constraints is
$2^{O(n)} \log m$ and hence the number of vertices is at most $(2^{O(n)} \log
m)^n$. With these estimates, we can show high probability concentration of
the shadow size around its mean when $m \geq 2^{\Omega(n^3)}$. One important
technical aspect ignored above is that both the independence properties and
the worst-case upper bounds on each $K_i$ crucially relies only on
conditioning $A$ to be ``locally'' $\epsilon$-dense around $[w_i,w_{i+1}]$
(see~\Cref{def:xy-local} and \Cref{lem:concentration} for more details). 

\paragraph*{\bf Abstract Diameter Bounds to the Rescue} To allow tight
concentration to occur for $m = 2^{\Omega(n^2)}$, we adapt the above strategy
by successively following shortest paths instead of the shadow path on $W$.
More precisely, between the maximizer $v_i$ of $w_i$ and $v_{i+1}$ of
$w_{i+1}$, $0 \leq i \leq k-1$, we follow the shortest path from $v_i$ to
$v_{i+1}$ in the subgraph induced by the vertices $v$ of $P(A)$ satisfying
$\sprod{v}{w_{i+1}} \geq \sprod{v_i}{w_{i+1}}$. We now let $K_i$, $0 \leq i \leq
k-1$, denote the length of the corresponding shortest path. For such local
paths, one can apply the abstract Barnette--Larman style bound
of~\cite{EHRR10} to obtain much better worst-case bounds. Namely, we can show
$K_i \leq 2^{O(n)} \log m$, $0 \leq i \leq k-1$, instead of $(2^{O(n)} \log
m)^n$ (see \Cref{lem:shortest-path-is-local-rv}). Crucially, the exact same
independence and locality properties hold for these paths as for the shadow
paths, due to the generality of our main locality lemma
(\Cref{lem:local-hardest-part}). Furthermore, as these paths are only shorter
than the corresponding shadow paths, their expected sum is again upper
bounded by Borgwardt's bound. With the improved worst-case bounds, our
concentration estimates are sufficient to show that all paths indexed by
planes in the net $N$ have length $O(n^2 m^{\frac{1}{n-1}} + 2^{O(n)})$ with
high probability.

%
% \begin{restatable}[Diameter Upper Bound]{theorem}{ubpa}\label{thm:ubpa}
% Let $A = \{a_1,\dots,a_M\} \in \sfe$, where $M$ is Poisson with $\E[M] = m$,
% and $a_1,\dots,a_M$ are uniformly and independently distributed in $\sfe$.
% Then, we have that
%     \[
% %         \Pr[D(A) >
% %         O(n^2 m^{\frac{1}{n-1}} + n 4^n \log^2(n^n m) 4^n m^{\frac{1/2}{n-1}})]
% %         < O(1/m).
%           \Pr[\diam(P(A)) >
%           O(n^2 m^{\frac{1}{n-1}} + n^8 16^n)] < O(1/m).
%     \]
% \end{restatable}

\subsubsection{The Lower Bound}

For the lower bound, we first reduce to lower bounding the diameter of the
polar polytope $P(A)^\circ = Q(A)$, where we show that $\diam(P(A)) \geq
(n-1)(\diam(Q(A))-2)$ (see~\Cref{lem:rel-diam}). This relation holds as long
as $P(A)$ is a simple polytope containing the origin in its interior (which
holds with probability $1-2^{-\Omega(m)}$). To prove it, we show that given
any path between vertices $v_1,v_2$ of $P(A)$ of length $D$, respectively
incident to distinct facets $F_1,F_2$ of $P(A)$, one can extract a facet
path, where adjacent facets share an $n-2$-dimensional intersection (i.e., a
ridge), of length at most $D/(n-1)+2$. Such facet paths exactly correspond to
paths between vertices in $Q(A)$, yielding the desired lower bound.   

For $m \geq 2^{\Omega(n)}$, proving that $\diam(P(A)) \geq \Omega(n
m^{1/(n-1)})$ reduces to showing that $\diam(Q(A)) \allowbreak \geq m^{1/(n-1)}$ with
high probability. For the $Q(A)$ lower bound, we examine the length of paths
between vertices of $Q(A)$ maximizing antipodal objectives, e.g., $-e_1$ and
$e_1$. From here, one can one easily derive an $\Omega((m/\log
m)^{\frac{1}{n-1}})$ lower bound on the length of such a path, by showing
that every edge of $Q(A)$ has length $\epsilon := \Theta((\log
m/m)^{\frac{1}{n-1}})$ and that the vertices in consideration are at distance
$\Omega(1)$. This is a straightforward consequence of $Q(A)$ being tightly
sandwiched by a Euclidean ball, namely $(1-\eps^2/2) B_2^n \subseteq Q(A)
\subseteq B_2^n$ (\Cref{lem:round}) with high probability. This sandwiching
property is itself a consequence of the rows of $A$ being $\epsilon$-dense on
$\sfe$, as mentioned in the previous section.  

Removing the extraneous logarithmic factor (which makes the multiplicative
gap between our lower and upper bound go to infinity as $m \rightarrow
\infty$), requires a much more involved argument as we cannot rely on a
worst-case upper bound on the length of edges. Instead, we first associate
any antipodal path above to a continuous curve on the sphere from $-e_1$ to
$e_1$ (\Cref{lem:pathobjective}), corresponding to objectives maximized by vertices along the path. From
here, we decompose any such curve into $\Omega(m^{\frac{1}{n-1}})$ segments
whose endpoints are at distance $\Theta(m^{-1/(n-1)})$ on the sphere.
Finally, by appropriately bucketing the breakpoints (\Cref{lem:deterministicLB}) and applying a careful
union bound, we show that for any such curve, an $\Omega(1)$ fraction of the
segments induce at least $1$ edge on the corresponding path with overwhelming
probability (\Cref{thm:LBlargem}). For further details on the lower bound, including how we
discretize the set of curves, we refer the reader to \Cref{sec:lower-bound}.  

% \begin{restatable}[Diameter Relation]{lemma}{diamrel}
% \label{lem:rel-diam}
% For $n \geq 2$, let $P \subseteq \R^n$ be a simple bounded polytope
% containing the origin in its interior and let $Q = P^\circ := \{x \in \R^n:
% \langle x, y \rangle \leq 1, \forall y \in P\}$ denote the polar of $P$.
% Then, $\diam(P) \geq (n-1)(\diam(Q)-2)$. 
% \end{restatable}
% 
% \begin{restatable}[Lower Bound for Q(A)]{theorem}{lbqa}
% \label{thm:LBlargem}
%     There exist positive constants $c_2<1$ and $c_3>1$ independent of $n \geq 3$ and $m$ such that the following holds. Let $A = \{a_1,\dots,a_M\} \in \sfe$, where $M$ is Poisson with $\E[M] = m$,
% and $a_1,\dots,a_M$ are uniformly and independently distributed in $\sfe$.
%    Then, with probability at least $1- e^{- c_3^{n-1} m^{1/(n-1)}} $, the combinatorial diameter of $Q(A)$ is at least $c_2 m^{1/(n-1)}$.
% \end{restatable}

\subsection{Organization}
In
\Cref{sec:prelims}, we introduce some basic notation as well as
background materials on Poisson processes (\Cref{subsec:poisson}), the measure of spherical caps (\Cref{subsec:cap_vol}), and
concentration inequalities for independent random variables (\Cref{subsec:nearly_iid}).
In \Cref{sec:upper-bound}, we prove the upper bound. Halfway into that section,
we also prove \Cref{thm:shadowbound}, a tail bound on the shadow size that is
of independent interest. We prove the lower bound in \Cref{sec:lower-bound}.

\section{Preliminaries}
\label{sec:prelims}
For notational simplicity in the sequel, it will be convenient to treat $A$
as a subset of $\sfe$ instead of a matrix. For $A \subseteq \sfe$, we will
slightly abuse notation and let $P(A) := \{x \in \R^n: \langle x, a \rangle
\leq 1, \forall a \in A\}$ and $Q(A) := {\rm conv}(A)$.
We denote the indicator of a random event $X$ by $1[X]$, i.e., $1[X] = 1$ if $X$
and $1[X] = 0$ otherwise.

For completeness sake, we first define paths and diameters.
\begin{definition}
    For any polyhedron $P \subset \R^n$, a path is a sequence $v_1,v_2,\dots,v_k \in P$
    of vertices, such that each line segment $[v_i, v_{i+1}], i \in [k-1],$ is an edge of $P$.
    A path is monotone with respect to an inner product $\sprod w \cdot$
    if $\sprod{w}{v_{i+1}} \geq \sprod{w}{v_i}$ for every $i \in [k-1]$.

    The distance between vertices $v_1, v_2 \in P$ is the minimum number $k$
    such that there exists a path $v_1',v_2',\dots,v_{k+1}'$ with $v_1 = v_1'$
    and $v_{k+1}' = v_2$. The diameter of $P$ is the maximal distance
    between any two of its vertices.
\end{definition}

\subsection{Density Estimates}
\label{sec:dense}

In this section, we give bounds on the fineness of the net induced by a
Poisson distributed subset of $\sfe$. Roughly speaking, if $A$ is
$\Pois(\sfe, m)$ distributed then $A$ will be $\Theta((\log
m/m)^{1/(n-1)})$-dense, see \Cref{def:dense}. While this estimate is standard in the stochastic
geometry, it is not so easy to find a reference giving quantitative
probabilistic bounds, as more attention has been given to establishing exact
asymptotics as $m \rightarrow \infty$ (see~\cite{RS16}). We
provide a simple proof of this fact here, together with the probabilistic
estimates that we will need.

\begin{definition}
    \label{def:dense}
    For $w \in \sfe$ and $r \geq 0$, we denote by $C(w,r) = \{x \in \sfe : \norm{w-x} \leq r\}$ the spherical cap of radius $r$ centered at $w$.

    We say $A \subseteq \sfe$ is $\eps$-dense in the sphere for $\eps > 0$ if
    for every $w \in \sfe$ there exists $a \in A$ such that $a \in C(w, \eps)$.
\end{definition}

\begin{lemma}
\label{lem:density}
For $m \geq n \geq 2$ and $0 < p < m^{-n}$, have
$\eps = \eps(m,n,p) > 0$ satisfy $\sigma(C(v,\eps)) = 3e\log(1/p)/m < 1/12$.
Then, for $A \sim \Pois(\sfe, m)$,
    \[
        \Pr[\exists v \in \sfe :\; C(v,\eps)\cap A = \emptyset ] \leq p
    \]
    and for every $t \geq 1$,
    \[
        \Pr[\exists v \in \sfe :\; \abs{C(v,t\eps)\cap A} \geq 45 \log(1/p) t^{n-1}] \leq p.
    \]
\end{lemma}
\begin{proof}%[Proof of Lemma \ref{lem:density}]
	\label{proof:density}
	Let $N \subset \sfe$ denote the centers of a maximal packing of spherical caps
	of radius $\eps/(2n)$. By maximality, $N$ is $\eps/n$-dense, i.e., an $\eps/n$ net.
	Comparing volumes, by Lemma~\ref{lem:cap-area}, we see that
	\[
	1 \geq \abs N \sigma(C(v,\eps/(2n)) \geq \abs N (2n)^{-(n-1)} \sigma(C(v,\eps)),
	\]
	so $\abs N \leq (2n)^{n-1} /\sigma(C(v,\eps))  \leq (2n)^{n-1} m$.
	By way of a net argument, using that $|C(v,(1-1/n)\eps) \cap A| \sim \Pois(m \sigma(C(v,(1-1/n)\eps))$, $\forall v \in \sfe$, we analyze our first probability
	\begin{align*}
		\Pr[\exists v \in \sfe :\; C(v,\eps)\cap A = \emptyset]
		&\leq \Pr[\exists v \in N :\; C(v,(1-1/n)\eps)\cap A = \emptyset] \\
		&\leq \abs N \max_{v \in N} \Pr[C(v,(1-1/n)\eps)\cap A = \emptyset] \\
		&\leq (2n)^{n-1} m e^{-m \sigma(C(v,(1-1/n)\eps))} \\
		&\leq (2n)^{n-1} m e^{-(1-1/n)^{n-1} m\sigma(C(v,\eps))} \\
		&\leq (2n)^{n-1} m e^{-3\log(1/p)} \leq p.
	\end{align*}
	We now prove the second estimate.
	By Lemma~\ref{lem:cap-area}, we have that
	$m\sigma(C(v,(1+1/n)t\eps)) \leq (1+1/n)^{n-1} t^{n-1} m \sigma(C(v,\eps))
	\leq 3e^2 t^{n-1} \log(1/p)$.
	Write $\lambda := 3e^2 t^{n-1}\log(1/p)$.
	By
	a similar net argument as above, we see that
	
	\begin{align*}
		\Pr[\exists v \in \sfe :\; \abs{C(v,t\eps)\cap A} \geq 2\lambda]
		&\leq \Pr[\exists v \in N :\; \abs{C(v,(1+1/n) t\eps)\cap A} \geq 2\lambda] \\
		&\leq \abs N \max_{v \in N} \Pr[\abs{C(v,(1+1/n) t\eps)\cap A} \geq 2\lambda] \\
		&\leq \abs N \Pr_{X \sim \Pois(\lambda)}[X \geq 2\lambda] 
		\leq \abs N e^{- \frac{\big(2\lambda - m\sigma(C(v,(1+1/n)t\eps))\big)^2}{4\lambda}} \\ & \quad \left(\text{ by the Poisson tailbound, Lemma~\ref{lem:pois-tail} }\right) \\ &
		\leq \abs N e^{- \frac{\lambda}{4} }
		\leq (2n)^{n-1} m e^{- 3\log(1/p)}
		\leq p.
	\end{align*}
	The proof is complete when we observe that $2\lambda \leq 45 t^{n-1} \log(1/p)$.
\end{proof}

We now give effective bounds on the density estimate $\eps$ above. Note that
taking the $(n-1)^{th}$ root of the bounds for $\eps^{n-1}$ below yields
$\eps = \Theta((\log m / m)^{1/(n-1)})$ for $m = n^{\Omega(1)}$ and $p =
1/m^{-n}$. The stated bounds follow directly from the cap measure estimates
in Lemma~\ref{lem:precise-cap}. 

\begin{corollary}\label{cor:caps}
Let $\eps > 0$ be as in Lemma~\ref{lem:density}, i.e., satisfying
$\sigma(C(v,\eps)) = 3e\log(1/p)/m \leq 1/12$.
    Then $\eps\in [0,\sqrt{2(1-\frac{2}{\sqrt{n}})}]$ and
    \[
        12e\log(1/p)/m \leq \eps^{n-1} \leq (\sqrt 2)^{n-1} \cdot 18\sqrt{n}\log(1/p)/m.
    \]
\end{corollary}
\begin{proof}%[Proof of Corollary \ref{cor:caps}]
\label{proof:caps}
The claim $\eps \in [0,\sqrt{2(1-\frac{2}{\sqrt{n}})}]$ follows by
Lemma~\ref{lem:precise-cap} part $1$ and our assumption that $\sigma(C(v,\eps))
\leq 1/12$. The lower bound on $\eps^{n-1}$ follows from the upper bound from
Lemma~\ref{lem:precise-cap} part $2$
\[
3e\log(1/p)/m = \sigma(C(v,\eps)) \leq \frac{1}{2(1-\eps^2/2)\sqrt{n}} (\eps \sqrt{1-\eps^2/4})^{n-1} \leq \frac{\eps^{n-1}}{4},
\]
where the last inequality follows since $\eps \in
[0,\sqrt{2(1-\frac{2}{\sqrt{n}})}]$. For the upper bound on $\eps$, we rely on the corresponding estimate in Lemma~\ref{lem:precise-cap} part $2$:
\[
3e\log(1/p)/m = \sigma(C(v,\eps))
\geq \frac{(\eps \sqrt{1-\eps^2/4})^{n-1}}{6(1-\eps^2/2)\sqrt{n}}
\geq \frac{(\eps \sqrt{1-\eps^2/4})^{n-1}}{6\sqrt{n}}
\geq \frac{(\eps/\sqrt{2})^{n-1}}{6\sqrt{n}} ,
\]
where the last inequality follows from $\eps \in [0,\sqrt{2}]$. The desired
inequalities now follow by rearranging.
\end{proof}

\subsection{Cap Volumes}\label{subsec:cap_vol}

For a subset $C \subseteq \sfe$, we write $\sigma(C) := \sigma_{n-1}(C)$ to
denote the measure of $C$ with respect to the uniform measure on $\sfe$. In
particular, $\sigma(\sfe) = 1$. For $v \in \sfe, \eps \geq 0$, let
$C(v,\eps) := \{x \in \sfe: \norm{x-v} \leq \eps\}$ denote the spherical
cap of radius $\eps$ around $v$.   
Throughout the paper, $\|\cdot\|$ will denote the Euclidean norm.

We will need relatively tight estimates on the measure of spherical caps.
The following lemma gives useful upper and lower bounds on the ratio of cap
volumes.

\begin{lemma}\label{lem:cap-area}
	For any $s, \eps > 0$ and $v \in \sfe$ we have
	\[
	\frac{\sigma(C(v,(1+s)\eps))}{(1+s)^{n-1}} \leq \sigma(C(v,\eps)) \leq \frac{\sigma(C(v,(1-s)\eps))}{(1-s)^{n-1}},
	\]
	assuming for the first inequality that $(1+s)\eps \leq 2$ and for the second that $s < 1$
	and $\eps \leq 2$.
\end{lemma}

\begin{proof}[Proof of Lemma \ref{lem:cap-area}]
	\label{proof:cap-area}
	First we write the area of the cap as the following integral, for any $r \in [0,2]$
	\[
	\sigma(C(v,r)) = c_{n-1} \int_0^{r^2/2} \sqrt{2t-t^2}^{n-3} \rm d t,
	\]
	where $c_{n-1} := {\rm vol}_{n-2}(\mathbb{S}^{n-2})/{\rm vol}_{n-1}(\sfe)$. Note that $\sqrt{2t-t^2}$ is the radius of the slice $\sfe \cap \{x \in \sfe: \langle x,v \rangle = 1-t\} = (1-t) v + \sqrt{2t-t^2} (S^{n-1} \cap v^\perp)$. The scaling of the volume of the central slice by $\sqrt{2t-t^2}^{n-3}$ instead of $\sqrt{2t-t^2}^{n-2}$ is to account for the curvature of the sphere. With this integral in our toolbox, we can prove our desired inequalities.
	We start with the first one, assuming that $(1+s)^2r^2/2 \leq 2$ so
	that we only take square roots of positive numbers.
	\begin{align*}
		\sigma(C(v,(1+s)\eps))
		&= c_{n-1} \int_0^{(1+s)^2r^2/2} \sqrt{2t-t^2}^{n-3}\rm d t \\
		&= c_{n-1} (1+s)^2 \int_0^{r^2/2}\sqrt{2(1+s)^2u - (1+s)^4u^2}^{n-3}\rm d u\\
		&\leq c_{n-1} (1+s)^2 \int_0^{r^2/2}\sqrt{2(1+s)^2u - (1+s)^2u^2}^{n-3}\rm d u\\
		&= (1+s)^{n-1} c_{n-1} \int_0^{r^2/2} \sqrt{2u-u^2}^{n-3}\rm d u \\
		&= (1+s)^{n-1} \sigma(C(v,\eps)).
	\end{align*}
	The second inequality is proven in a similar fashion, assuming that $1-s
	> 0$:
	\begin{align*}
		\sigma(C(v,(1-s)\eps))
		&= c_{n-1} \int_0^{(1-s)^2r^2/2} \sqrt{2t-t^2}^{n-3}\rm d t \\
		&= c_{n-1}(1-s)^2 \int_0^{r^2/2}\sqrt{2(1-s)^2t - (1-s)^4t^2}^{n-3}\rm d t\\
		&\geq c_{n-1}(1-s)^2 \int_0^{r^2/2}\sqrt{2(1-s)^2t - (1-s)^2t^2}^{n-3}\rm d t\\
		&= (1-s)^{n-1} c_{n-1}\int_0^{r^2/2} \sqrt{2t-t^2}^{n-3}\rm d t \\
		&= (1-s)^{n-1} \sigma(C(v,\eps)).
	\end{align*}
\end{proof}

We now give absolute estimates on cap volume measure due to~\cite{BGKKLS01}.
We note that~\cite{BGKKLS01} parametrize spherical caps with respect to the
distance of their defining halfspace to the origin. The following lemma is
derived using the fact that the cap $C(v,\eps)$, $\eps \in [0,\sqrt{2}]$, $v
\in \sfe$, is induced by intersecting $\sfe$ with the halfspace $\langle v, x
\rangle \geq 1-\eps^2/2$, whose distance to the origin is exactly
$1-\eps^2/2$. 

\begin{lemma}\cite[Lemma 2.1]{BGKKLS01}
	\label{lem:precise-cap}
	For $n \geq 2$, $\eps\in [0,\sqrt{2}]$, $v \in \sfe$, the following
	estimates holds:
	\begin{itemize}
		\item If $\eps\in [\sqrt{2(1-\frac{2}{\sqrt{n}})},\sqrt{2}]$, then $\sigma(C(v,\eps)) \in [1/12,1/2]$.
		\item If $\eps\in [0,\sqrt{2(1-\frac{2}{\sqrt{n}})}]$, then
		\[
		\frac{1}{6(1-\eps^2/2)\sqrt{n}} (\eps\sqrt{1-\eps^2/4})^{n-1} \leq \sigma(C(v,\eps)) \leq \frac{1}{2(1-\eps^2/2)\sqrt{n}}(\eps\sqrt{1-\eps^2/4})^{n-1}.
		\]
	\end{itemize}
\end{lemma}

\subsection{Poisson Processes}\label{subsec:poisson}

The Poisson distribution $\Pois(\lambda)$ with parameter $\lambda \geq 0$ has
probability mass function $f(x,\lambda) := e^{-\lambda}
\frac{\lambda^x}{x!}$, $x \in \Z_+$. We note that $\Pois(0)$ is the random
variable taking value $0$ with probability $1$. Recall that
$\E[\Pois(\lambda)] = \lambda$. We will rely on the following standard
tail-estimate (see~\cite[Theorem 1]{CL19}):

\begin{lemma}
	\label{lem:pois-tail}
	Let $X \sim \Pois(\lambda)$. Then for $x \geq 0$, we have that
	\begin{equation}
		\max \{ \Pr[X \geq \lambda+x], \Pr[X \leq \lambda-x] \} \leq e^{-\frac{x^2}{2(\lambda+x)}}.
	\end{equation}
\end{lemma}  

We define a random subset $A$ to be distributed as $\Pois(\sfe,\lambda)$,
$\lambda \geq 0$, if $A = \{a_1,\dots,a_M\}$, where $|A| = M \sim
\Pois(\lambda)$ and $a_1,\dots,a_M$ are uniformly and independently
distributed on $\sfe$.  Note that $\E[|A|] = \lambda$. In standard
terminology, $A$ is called a homogeneous Poisson point process on $\sfe$ with
intensity $\lambda > 0$. 

A basic fact about such a Poisson process
is that the number of samples landing in disjoint subsets are independent
Poisson random variables.   

\begin{proposition} 
	\label{prop:pois-ind}
	Let $A \sim \Pois(\sfe,\lambda)$. Let $C_1,\dots,C_k \subseteq \sfe$ be
	pairwise disjoint measurable sets. Then, the random variables $|A \cap C_i|$,
	$i \in [k]$, are independent and $|A \cap C_i| \sim \Pois(\lambda \sigma(C_i))$,
	$i \in [k]$.
\end{proposition}

\subsection{Concentration for Nearly-Independent Random Variables}\label{subsec:nearly_iid}

For a random variable $X \in \R$, let $\Var[X] := \E[X^2]-\E[X]^2$ denote its
variance. 

We will use the following variant on Bernstein's inequality that is a direct
consequence of \cite[Theorem 2.3]{janson}, which proves a more general result
using the fractional chromatic number of the dependency graph.

\begin{lemma}\label{lem:bernstein} Suppose that $Y_1,\dots,Y_k$ are random variables taking values in $[0, M]$
	and $\mathrm{Var}(Y_i) \leq \sigma^2$ for each $i \in [k]$.
	Assume furthermore that there exists a partition $I_1 \cup I_2 \cup \dots
	\cup I_q = \{Y_1,\dots,Y_k\}$ such that the random variables in any one set $I_j$
	are mutually independent.
	Then for any $t \geq 0$ we get
	\[
	\Pr\left[\abs*{\sum_{i=1}^k Y_i - \E[\sum_{i=1}^k Y_i]} \geq t\right]
	\leq 2\exp\left(\frac{-8t^2}{25q(k\sigma^2 + Mt/3)}\right)
	\]
\end{lemma}

When we use the above lemma, we will bound the variance
of the random variables using the following inequality:
\begin{lemma}\label{lem:bhatia-davis}
	Let $Y \in [0,M]$ be a random variable and $\E[Y] = \mu$.
	Then $\mathrm{Var}(Y) \leq \mu (M-\mu)$.
\end{lemma}
\begin{proof}
	The inequality follows from $\mathrm{Var}(Y) = \E[Y^2] - \mu^2 \leq M
	\E[Y]-\mu^2 = \mu(M-\mu)$, where we have used that $Y^2 \leq MY$ for $Y \in
	[0,M]$. 
\end{proof}

\section{Shadow size and upper bounding the diameter}
\label{sec:upper-bound}
In the first part of this section, we prove a concentration result
on the number of \emph{shadow vertices} of $P(A)$.
This addresses an open problem from \cite{Borgwardt87}.
In the second part, we use the resulting tools to prove \Cref{thm:ubpa}, our
high-probability upper bound on the diameter of $P(A)$. We start by defining a
useful set of paths for which we know their expected lengths.
\begin{definition}
    Let $P \subset \R^n$ be a polyhedron and $W \subset \R^n$
    be a two-dimensional linear subspace.
    We denote by $\mathcal S(P, W)$ the set of \emph{shadow vertices}: the
    vertices of $P$ that maximize a non-zero objective function $\sprod w \cdot$ with
    $w \in W$.
\end{definition}

From standard polyhedral theory, we get a characterization of shadow vertices:
\begin{lemma}\label{lem:shadowpathcharacterization}
    Let $P(A)$ be a polyhedron given by $A \subset \R^n$ and $w \in \R^n \setminus \{0\}$.
    A vertex $v \in P(A)$ maximizes $\sprod{w}{\cdot}$
    if and only if $w\R_+ \cap \conv\{a \in A : \sprod a v = 1\} \neq \emptyset$.

    Hence for $W \subset \R^n$ a two-dimensional linear subspace,
    a vertex $v \in P(A)$ is a shadow vertex $v \in \mathcal S(P(A),W)$
    if and only if $\conv\{a \in A : \sprod a v = 1\} \cap W \setminus \{0\} \neq \emptyset$.
\end{lemma}

The set of shadow vertices for a fixed plane $W$ induces a connected subgraph
in the graph consisting of vertices and edges of $P$, and so any two shadow
vertices are connected by a path of length at most $\abs{\mathcal S(P,W)}$.
As such, for nonzero $w_1, w_2 \in W$, we might speak of a \emph{shadow path}
from $w_1$ to $w_2$ to denote a path from a maximizer of $\sprod {w_1} \cdot$
to a maximizer of $\sprod {w_2} \cdot$ that stays inside $\mathcal S(P,W)$
and is monotonous with respect to $\sprod {w_2} \cdot$.
The shadow path was studied by Borgwardt:
\begin{theorem}[\cite{Borgwardt87,Borgwardt99}]\label{thm:borgwardt}
    Let $m \geq n$ and fix a two-dimensional linear subspace $W \subset \R^n$.
    Pick any probability distribution on $\R^n$ that is
    invariant under rotations and let the entries of $A \subset \R^n, 
    \abs A = m$, be independently sampled from this distribution.
    Then, almost surely, for any linearly independent $w_1,w_2 \in W$ there is a unique shadow
    path from $w_1$ to $w_2$. Moreover, the vertices in $\mathcal S(P(A),W)$ are in one-to-one
    correspondence to the vertices of $\pi_W(P(A))$, the orthogonal projection of $P(A)$ onto $W$.
    The expected length of the shadow path from $w_1$ to $w_2$ is at most
    \[\E[\abs{\mathcal S(P(A), W)}] = O(n^2 m^{\frac{1}{n-1}}).\]
    This upper bound is tight up to constant factors for the uniform distribution on $\sfe$.
\end{theorem}

We prove a tail bound for the shadow size when $A \sim \Pois(\sfe, m)$.
This result answers a question of Borgwardt in the asymptotic regime,
regarding whether bounds on higher moments of the shadow size can be given.
To obtain such concentration, we show that the shadow decomposes into a sum
of nearly independent ``local shadows'', using that $A$ will be $\eps$-dense
per \Cref{lem:density}, allowing us to apply standard
concentration results for sums of nearly independent random variables. 

\begin{restatable}[Shadow Size Concentration]{theorem}{ubshadow}\label{thm:shadowbound}
    Let $e^{\frac{-m}{18\sqrt{n} (76\sqrt{2})^{n-1}}} < p < m^{-2n}$ and let
    \[t_p := \max\left(\sqrt{O(U n^2  m^{\frac{1}{n-1}} \log(1/p))}, O(U \log(1/p))\right)\]
    for $U := O(n 2^{n^2} (\log(1/p))^n)$.
    If $A \sim \Pois(\sfe,m)$ then the shadow size satisfies
    \[
        \Pr\big[\Big|\abs{\mathcal S(P(A),W)} - \E[\abs{\mathcal S(P(A),W)}]\Big| > t_p\big]
            \leq 4p.
    \]
\end{restatable}

After that, we extend the resulting tools
to obtain our upper bound on the diameter.

\begin{restatable}[Diameter Upper Bound]{theorem}{ubpa}\label{thm:ubpa}
    Let $e^{\frac{-m}{18\sqrt{n} (76\sqrt{2})^{n-1}}} < p < m^{-2n}$.
If $A = \{a_1,\dots,a_M\} \in \sfe$, where $M$ is Poisson with $\E[M] = m$,
and $a_1,\dots,a_M$ are uniformly and independently distributed in $\sfe$.
Then, we have that
    \[
          \Pr[\diam(P(A)) >
          O(n^2 m^{\frac{1}{n-1}} + n 4^n \log^2 (1/p))] \leq O(\sqrt p).
    \]
\end{restatable}
\begin{proof}%[Proof of \Cref{thm:ubpa}]
	From \Cref{cor:caps}, for $\eps := \eps(m,n,p)$, we know that $\eps^{n-1}
	\leq \frac{1}{76^{n-1}}$ given the lower bound on $p$. In particular, $\eps <
	1/76$.
	
	Let $N \subset \sfe$ be a fixed minimal $\eps$-net.
	Consider the following statements:
	\begin{itemize}
		\item
		For every $n \in N$, any two vertices in $\mathcal S(P(A),\mathrm{span}(e_1,n))$
		are connected by a path of length at most $O(n^2 m^{\frac{1}{n-1}}) + t$,
		where $t$ is defined in \Cref{thm:shadow-plane-diam}.
		\item
		$A$ is $\eps$-dense.
		\item
		For any $x \in \sfe$ we have $\abs{A \cap C(x,(2+2/n)\eps)} \leq 45 e 2^n \log(1/p)$.
	\end{itemize}
	For given $n \in N$, the first event holds with probability at least $1-4p$
	by \Cref{thm:shadow-plane-diam}. The net $N$ has $|N| \leq (4/\eps)^n$
	points, which is at most $4^n \cdot m$ by \Cref{cor:caps}.
	By the union bound the first statement holds for all $n\in N$ simultaneously
	with probability at least $1-\sqrt p$.
	From \Cref{lem:density} we know that the second statement holds with
	probability at least $1-p$ and the third statement holds with probability
	at least $1 - p$. We conclude that all three statements hold simultaneously
	with probability at least $1-O(\sqrt p)$.
	
	We will show that the above conditions imply the bound on the
	combinatorial diameter of $P(A)$.
	
	First, observe that we only need to show an upper bound for all $w \in \sfe$
	on the length of a path connecting any vertex maximizing $\sprod w \cdot$
	to a vertex maximizing $\sprod{e_1}\cdot$.
	The combinatorial diameter of $P(A)$ is at most twice that upper bound.
	
	Let $w \in \sfe$ and pick $n \in N$ such that $\|w-n\| \leq \eps$.
	By the first statement, there is a path from
	the vertex maximizing $\sprod n \cdot$
	to the vertex maximizing $\sprod {e_1} \cdot$
	of length $O(n^2 m^{\frac{1}{n-1}}) + t$.
	
	By the second two statements, $E_{w_1,w_2}$ is satisfied
	for every $w_1,w_2\in\sfe$.
	We conclude from \Cref{lem:shortest-path-is-local-rv}
	that there is a path from
	any vertex maximizing $\sprod w \cdot$ to the vertex maximizing $\sprod n \cdot$
	of length $45e n 4^n \log(1/p)$.
	
	Therefore, when all three statements hold
	the combinatorial diameter of $P(A)$ is at most $O(n^2m^{\frac{1}{n-1}}) + t_p + 45e n 4^n \log(1/p)$. Now we fill in $t_p$ and obtain an upper bound of
	\[
	O(n^2 m^{\frac{1}{n-1}} + n4^n \log^2 p).
	\]
\end{proof}

\subsection{Only `nearby' constraints are relevant}
We will start by showing that, with very high probability,
constraints that are 'far away' from a given point on the sphere
will not have any impact on the local shape of paths.
That will result in a degree of independence between
different parts of the sphere, which will be essential in
getting concentration bounds on key quantities.

\begin{lemma}\label{lem:round}
    If $A \subset \sfe$ is $\eps$-dense for $\eps \in [0,\sqrt 2)$ then 
    \(
         \ball^n 
        \subset P(A)
        \subset \left(1-\frac{\eps^2}{2}\right)^{-1} \ball^n.
    \)
\end{lemma}
\begin{proof}%[Proof of Lemma \ref{lem:round}]
\label{proof:round}
The first inclusion follows immediately from the construction of $P(A)$.  We
now show the second inclusion. Taking $x \in P(A) \setminus \{0\}$, we must
show that $\norm{x} \leq (1-\eps^2/2)^{-1}$. For this purpose, choose $a \in
A$ such that $\norm{a-x/\norm{x}} \leq \eps$, which exists by our assumption
that $A$ is $\eps$-dense. Since $\eps^2 \geq \norm{a-x/\norm{x}}^2 =
2(1-\sprod{a}{x/\norm{x}})$, we have that $\sprod{a}{x/\norm{x}} \geq
1-\eps^2/2$. Since $x \in P(A)$, we have $1 \geq \sprod{a}{x} \geq
(1-\eps^2/2)\norm{x}$, and the bound follows by rearranging. 
\end{proof}

\begin{lemma}\label{lem:norm-sprod-distance}
    If $w \in \sfe$, $\alpha < 1$, $\|v\| \leq (1-\alpha)^{-1}$ and $\sprod v w \geq 1$
    then $\|v/\|v\| - w\|^2 \leq 2\alpha$.
\end{lemma}
\begin{proof}
    We have $1 \leq \sprod v w = \|v\| \cdot \sprod{v/\|v\|}{w} \leq
    (1-\alpha)^{-1}\sprod{v/\|v\|} w$.  Hence $1 - \|v/\|v\| - w\|^2/2 = \sprod
    {v/\|v\|} w \geq 1-\alpha$,
    which exactly implies that $\|v/\|v\| - w\|^2 \leq 2\alpha$ as required.
\end{proof}

We will use the above lemmas to prove the main technical estimate of this subsection:
if $A\subset\sfe$ is $\eps$-dense and $w_1,w_2 \in \sfe$ satisfy
$\norm{w_1 - w_2} \leq 2\eps/n$ then any vertex on any path on $P(A)$ starting
at a maximizer of $\sprod {w_1} \cdot $ that is
non-decreasing with respect to $\sprod {w_2} \cdot $ can only be tight at
constraints $\sprod a x = 1$ induced by $a \in A \cap C(w_2, (2+2/n)\eps)$.
All other constraints are strictly satisfied by every vertex on such a monotone path.
\begin{lemma}\label{lem:hardest-part}
    Let $\eps\in [0,1]$ and assume that $w_1 , w_2 \in \sfe$ satisfy
    $\norm{w_1 - w_2} \leq (1-\eps^2/2)$.
    Let $v_1, v \in \R^n$ satisfy $\sprod{w_1}{v_1} \geq 1$ and
    $\sprod{w_2}{v} \geq \sprod{w_2}{v_1}$,
    and assume $\norm{v_1}, \norm{v} \leq (1-\eps^2/2)^{-1}$.
    Last, let $a \in \sfe$ satisfy $\sprod a v \geq 1$.
    Then we have $\norm{w_2 - a} \leq 2\eps + \norm{w_1-w_2}$.
\end{lemma}
\begin{proof}%[Proof of Lemma \ref{lem:hardest-part}]
\label{proof:hardest-part}
\begin{figure}
	\centering
	\begin{tikzpicture}[every node/.style={fill,circle,inner sep=0pt,minimum size=4pt}]
		%%% Set paramaters
		% radius of the circle
		\def\radius{6};
		\def\radiuslarge{6.5};
		% clip the picture
		\clip ({-\radiuslarge},{0.55*\radius}) rectangle (\radiuslarge,\radiuslarge);
		% positions of w_1, w_2, v_1, a, v,
		\def\anglewone{78};
		\def\anglewtwo{90};
		\def\anglevone{61};
		\def\radiusvone{1.065*\radius}
		\def\anglea{139};
		\def\anglev{117};
		%%%% Draw
		%% Draw the circles
		\draw[name path=unitcircle] (0,0) circle [radius=\radius];
		\draw[dotted, name path=largecircle] (0,0) circle [radius=\radiuslarge];
		%% Draw w_1, w_2, a, v/||v||, v_1 
		\node[label={{\anglewone+105}:$w_1$}] (w1) at (\anglewone:\radius) {};
		\node[label={{\anglewtwo+170}:$w_2$}] (w2) at (\anglewtwo:\radius) {};
		\node[label={{\anglea+180}:$a$}] (a) at (\anglea:\radius) {};
		\node[label={{\anglev+180}:$\frac{v}{\norm{v}}$}] (vnormalized) at (\anglev:\radius) {};
		\node[label={{\anglevone}:$v_1$}] (v1) at (\anglevone:\radiusvone) {};
		%% Construction related to v'_1, 
		% Invisible tangent at w_1
		\path[name path=tangentwone] (w1) -- +({\anglewone-90}:\radius) ;
		% Draw v'_1
		\node [name intersections={of=largecircle and tangentwone},label={0:$v'_1$}] (voneprime) at (intersection-1) {};
		% Extract the angle coordinate of v'_1 and use it to compute its angle and then the angle of its symmetric with respect to w_1. This symmetric point is named vonesecond in this TikZ code.
		\pgfgetlastxy{\XCoord}{\YCoord};
		\def\anglevoneprimecompute{atan(\YCoord/\XCoord)};
		\def\anglevoneprime{{\anglevoneprimecompute}};
		\def\anglevonesecond{{2*\anglewone-\anglevoneprimecompute}};
		\coordinate (vonesecond) at (\anglevonesecond:\radiuslarge);
		% Fill B
		\draw[fill,fill opacity=0.5,draw opacity=0,thick,gray] (voneprime.center) -- (vonesecond.center) arc (\anglevonesecond:\anglevoneprime:\radiuslarge) ;
		% Draw tangent at w_1
		\draw[dashed] (w1) -- (voneprime);
		% Draw v'_1/||v'_1||
		\path[name path=linevoneprime] (0,0) -- (voneprime) ;
		\node [name intersections={of=unitcircle and linevoneprime},label={{\anglevone+180}:$\frac{v'_1}{\norm{v'_1}}$}] (voneprimenorm) at (intersection-1) {};
		% Draw horizontal line at the height of v'_1
		\draw[dashed] (v1) -- +({-2*\radius},0);
		\draw[dashed] (v1) -- +({2*\radius},0);
		%% Construction related to v
		% Invisible tangent at a
		\path[name path=tangenta] (a) -- +({\anglea-90}:\radius) ;
		% Draw v (created as the intersection of the tangent at a and the line passing through the origin and the normalisation of v)
		\path[name path=linev] (0,0) -- (\anglev:{1.2*\radius}) ;
		\node [name intersections={of=tangenta and linev},label={{\anglev}:$v$}] (v) at (intersection-1) {};
		% Draw tangent at a
		\draw[dashed] (a) -- (v);
	\end{tikzpicture}
	\caption{Illustration of the proof of Lemma \ref{lem:hardest-part}. The inner (resp. outer dotted) curve represents part of the sphere $\sfe$ (resp. $(1-\eps^2/2)^{-1}\sfe$). The horizontal dashed line represents the hyperplane $\{ x\in\R^d : \sprod{x}{w_2} = \sprod{v_1}{w_2}\}$. The two oblique dashed line segments represent parts of the hyperplanes tangent to the unit sphere at the points $a$ and $w_1$. The grey area represents the set $B$.}
	\label{fig:my_label}
\end{figure}
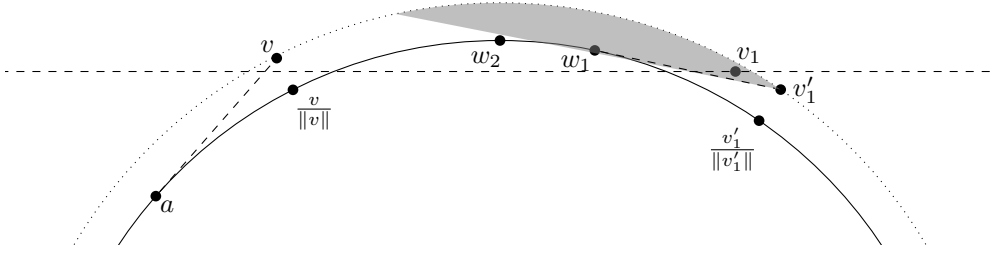

By \Cref{lem:norm-sprod-distance}, since $\sprod{w_1}{v_1}, \sprod{a}{v} \geq 1$
and $\norm{w_1}=\norm{a}=1$, we get that
$\norm{w_1-v_1/\norm{v_1}},\norm{a-v/\norm{v}} \leq \eps$.

If $w_1=w_2$, then by assumption $\sprod{v}{w_2} \geq \sprod{v_1}{w_2} =
\sprod{v_1}{w_1} \geq 1$. Thus, \Cref{lem:norm-sprod-distance} implies that
$\norm{w_2-v/\norm{v}} \leq \eps$. By the triangle inequality, we conclude
that $\norm{w_2-a} \leq \norm{w_2-v/\norm{v}}+\norm{v/\norm{v}-a} \leq
2\eps$, as needed. 

Now assume that $w_1 \neq w_2$. To prove the lemma, we show that it suffices to $v'_1$ such that the following two inequalities hold:
\begin{equation} \label{eq:GoalProofLemHard}
	\norm*{ \frac{v}{\norm{v}} - w_2 }
	\leq \norm*{ \frac{v'_1}{\norm{v'_1}} - w_2 } , \quad \quad
	\norm*{\frac{v'_1}{\norm{v'_1}} - w_1 }
	\leq \eps .
\end{equation}
Indeed, given $v'_1$ as above, the triangle inequality and the first inequality of \eqref{eq:GoalProofLemHard} imply that
\begin{align*}
	\norm{w_2 - a}
	&\leq \norm*{ w_2 - \frac{v}{\norm{v}} } + \norm*{ \frac{v}{\norm{v}} - a } \\
	& \leq \norm*{ w_2 - \frac{v'_1}{\norm{v'_1}} } + \norm*{ \frac{v}{\norm{v}} - a } \\
	& \leq \norm*{ w_2 - w_1 } + \norm*{ w_1 - \frac{v'_1}{\norm{v'_1}} } + \norm*{ \frac{v}{\norm{v}} - a }.
\end{align*}
From here, by the second inequality of \eqref{eq:GoalProofLemHard} and
$\norm{a-v/\norm{v}} \leq \eps$, we get that
\begin{align*}
	\norm{w_2 - a} 
	& \leq \norm{w_2-w_1} + \eps + \eps ,
\end{align*}
which is the claim of the lemma. To construct $v'_1$, let
\[ B := \left\{ x \in \R^d : 
\sprod{w_1}{x} \geq 1,\ 
\norm{x} \leq \left( 1-\frac{\eps^2}{2} \right)^{-1} 
\right\} .\]
and define $v'_1$ to be the minimizer of $\sprod{w_2}{\cdot}$ in $B$. Since
$w_1 \neq w_2$, it is direct to verify the $v'_1$ is unique and satisfies
$\norm{v'_1} = (1-\frac{\eps^2}{2})^{-1}$.

%     We now $v'_1 \in B$  $\sprod{w_2}{\cdot}$ is minimal.
%     Note that the minimum is achieved at the boundary of $B$ because
%     $\sprod{w_2}{\cdot}$ is a linear function and $B$ is a compact set.
%     Since $w_1, w_2$ are linearly independent , $v_1'$ cannot be only
%     tight at the constraint $\sprod{w_1}{x} \geq 1$.
%     This implies that $\|v_1'\| = (1-\eps^2/2)^{-1}$.

\bigskip
From \Cref{lem:norm-sprod-distance} we have that any point
$x \in B$ satisfies $\|x/\|x\| - w_1\| \leq \eps$,
and in particular this is true for $v_1'$,
making the second inequality of \eqref{eq:GoalProofLemHard} hold. Note that $v_1 \in B$ as well. It remains to show the first inequality of \eqref{eq:GoalProofLemHard}.
For this, we claim that
\[
\sprod{w_2}{v} \geq \max\{0,\sprod{w_2}{v'_1}\},
\]
By assumption, recall that $\sprod{w_2}{v} \geq \sprod{w_2}{v_1}$. 
The first inequality now follows since $\sprod{w_2}{v_1} \geq \sprod{w_1}{v_1} -
\|w_1-w_2\|\|v_1\| \geq 1 - \|w_1-w_2\|(1-\eps^2/2)^{-1} \geq 0$, by our
assumption on $\norm{w_1-w_2}$. The second inequality now follows from 
$\sprod{w_2}{v_1} \geq \sprod{w_2}{v'_1}$, which holds since $v_1 \in B$ and $v'_1$ minimizes $w_2$ over $B$. 

Using that $\norm{v} \leq (1-\eps^2/2)^{-1} = \norm{v_1'}$, we conclude that
\begin{align*}
	\sprod{w_2}{\frac{v}{\norm{v}}} \geq \sprod{w_2}{\frac{v}{\norm{v'}}} \geq \sprod{w_2}{\frac{v'_1}{\norm{v'_1}}} ,
\end{align*}
where the first inequality uses $\sprod{w_2}{v} \geq 0$. The first
inequality of \eqref{eq:GoalProofLemHard} now follows from the fact that $ u
\in\sfe \mapsto \norm{u - w_2} $ is a decreasing function of
$\sprod{u}{w_2}$, and thus the proof is complete. 
\end{proof}

To round out this subsection, we prove that the conclusion of \Cref{lem:hardest-part}
holds whenever $v, v_1 \in P(A)$ and $A$ is $\eps$-dense in a
neighbourhood around $w_2$.
\begin{definition}
    Given sets $A, C \subseteq \sfe$ and $\eps > 0$, we say that
    $A$ is $\eps$-dense for $C$ if for every $c \in C$ there exists
    $a \in A$ such that $\|a-c\| \leq \eps$.
\end{definition}

\begin{lemma}\label{lem:local-hardest-part}
    Let $A\subset \sfe$ be compact and $\eps$-dense for $C(w_2,4\eps)$, $\eps > 0$.
    Let $v_1, v \in P(A)$ and $w_1, w_2 \in \sfe$ satisfying $\sprod {w_1}{v_1} \geq 1$,
    $\sprod{w_2}{v} \geq \sprod{w_2}{v_1}$ and $\norm{w_1-w_2} \leq \eps$.
    Now let $a \in \sfe$ satisfy $\sprod {a}{v} \geq 1$.
    Then we have $\|v_1\|, \|v\| \leq (1-\eps^2/2)^{-1}$ and
    $\|w_2 - a \| \leq 2\eps + \norm{w_1-w_2}$.
\end{lemma}

Note also the contrapositive of the above statement: for $w_1,w_2, v_1, v, A$
satisfying the conditions above, we have for $a \in \sfe$ that $\|w_2 - a\| >
2\eps + \norm{w_1-w_2}$ implies $\sprod a v < 1$.

\begin{proof}[Proof of Lemma \ref{lem:local-hardest-part}]
	\label{proof:local-hardest-part}
	First, observe that if $\eps \geq 1$ then the conclusion is trivially satisfied
	since $\|w_2 - a\| \leq 2 \leq 2\eps + \norm{w_1 - w_2}$. From now on, assume $\eps < 1$.
	
	Let $A \subset A' \subset \sfe$ be $\eps$-dense,
	such that $A' \cap C(w_2, 3\eps) \subseteq A$.
	One valid choice is to take any $\eps$-net $N \subset \sfe$
	and set $A' = A \cup (N \setminus C(w_2, 3\eps))$.
	Then any $x \in C(w_2, 4\eps)$ has an
	$a \in A \subset A'$ with $\|a-x\| \leq \eps$
	and any $y \notin C(w_2, 4\eps)$ has some
	$b \in N$ with $\|y-b\| \leq \eps$
	and $b \notin C(w_2, 3\eps)$.
	Moreover we have
	$(N \setminus C(w_2,3\eps)) \cap C(w_2, 3\eps) = \emptyset$
	so this choice of $A'$ satisfies our requirements.

	If $v, v_1 \in P(A')$, then $\norm{v}, \norm{v_1} \leq (1-\eps^2/2)^{-1}$ by
	\Cref{lem:round} and we can apply \Cref{lem:hardest-part}
	to the set $A'$ and vectors $w_1, w_2, v, v_1$ and $a$
	to conclude $\|w_2 - a\| \leq 2\eps + \norm{w_1-w_2}$ as required.
	
	We now prove that both the case $v_1 \notin P(A')$
	and the case $v_1 \in P(A'), v \notin P(A')$ lead to contradiction.
	First, observe that given $w_1$ and $w_2$,
	the set of pairs $(v_1, v)$ that satisfy $\sprod{w_1}{v_1} \geq 1,
	\sprod{w_2}{v} \geq \sprod{w_2}{v_1}$ and
	$\norm{v_1},\norm{v} \leq (1-\eps^2/2)^{-1}$
	is a closed convex set and contains $(w_1,w_1)$.

	If $v_1 \notin P(A')$,
	let $(x, y)$ be the convex combination of $(v_1, v_1)$
	and $(w_1,w_1)$ such that $x = y \in P(A')$
	and there exists $a' \in A' \setminus A$ such that $\sprod{a'}{x} = 1$.
	Such $a'$ will exist because $A'$ is compact.
	
	Otherwise we have $v \notin P(A')$
	and let $(x, y)$ be a convex combination of $(v_1, v)$
	and $(w_1, w_1)$
	such that $x, y \in P(A')$ and there exists $a' \in A' \setminus A$
	such that $\sprod{a'}x = 1$.
	Such $a'$ will exist because $A'$ is compact.
	
	Either way, apply \Cref{lem:hardest-part} to $A', w_1, w_2, x, y$ and $a'$
	to find that $\|w_2 - a'\| \leq 2\eps + \norm{w_1-w_2}$.
	This contradicts the earlier claim that $a' \in A' \setminus A$.
	From this contradiction we conclude that $v, v_1 \in P(A')$,
	which finishes the proof.
\end{proof}

\subsection{Locality, independence, and concentration}

With an eye to \Cref{lem:local-hardest-part},
this subsection is concerned with proving concentration
for sums of random variables that behave nicely when $A$ is dense
in given neighbourhoods.
The specific random variables that we will use this for
are the paths between the maximizers of nearby objective
vectors $w_1, w_2 \in \sfe$.

\begin{definition}
\label{def:xy-local}
    Given $m,n,p,$ let $\eps = \eps(m,n,p) > 0$
    be as in \Cref{lem:density}
    and $A \subset \R^n$ be a random set.
    For $x, y \in \sfe$ define the event $E_{x,y}$ as:
    \begin{itemize}
        \item $A$ is $\eps$-dense for $C(x, \|x - y\| + 4\eps)$, and
        \item for every $z \in [x, y]$ we have
            \[\abs*{A \cap C(\frac{z}{\|z\|}, (2+2/n)\eps)} \leq 45e 2^n \log(1/p)\]
    \end{itemize}
    A random variable $K$ is called $(x,y)$-local if
    $E_{x,y}$ implies that $K$ is a function of
    $A \cap C(x, 5\eps + \|x-y\|)$.
\end{definition}
In particular, we will use that if $K$ is $(x,y)$-local then
$K1[E_{x,y}]$ is a function of $A \cap C(x,5\eps + \|x-y\|)$.

To help prove that certain path are local random variables, we will use the following helper lemma. 
\begin{lemma}\label{lem:geodesic}
    Let $w_1,w_2 \in \sfe$, and have $w_1 = v_1,v_2,\dots,v_{k+1}=w_2$
    be equally spaced on a shortest geodesic segment on $\sfe$ connecting $w_1$ and $w_2$.
    Then for every $i \in [k]$ we have
    $\|w_1 - w_2\|/4k \leq \|v_i - v_{i+1}\| \leq \|w_1 - w_2\|/k$.
\end{lemma}
\begin{proof}%[Proof of Lemma \ref{lem:geodesic}]
\label{proof:geodesic}
By the triangle inequality, we have $\|w_1 - w_2\| \leq \sum_{i=1}^k \|v_i - v_{i+1}\|$.
Since each of the line segments $[v_i, v_{i+1}]$ has identical length,
this gives us the second inequality $\|v_i - v_{i+1}\| \leq \|w_1 - w_2\|/k$.

Furthermore, we know that the geodesic segment connecting $w_1$ and $w_2$
has length at most $\pi \|w_1 - w_2\|$. From this we get
$\sum_{i=1}^k \|v_i - v_{i+1}\| \leq \pi\|w_1 - w_2\|$
and hence $\|w_1 - w_2\|/4k \leq \|w_1 - w_2\|/\pi k \leq \|v_i - v_{i+1}\|$.
\end{proof}

Many paths on $P(A)$ turn out to be such local random variables.
One example are short segments of the shadow paths from \Cref{thm:borgwardt}.
\begin{lemma}\label{lem:shadow-path-is-local-rv}
    Let $w_1, w_2 \in \sfe$ satisfy $\|w_1 - w_2\| \leq \eps$.
    Then the length of the shadow path on $P(A)$ from $w_1$ to $w_2$
    is a $(w_1, w_2)$-local random variable.
    If $\|w_1 - w_2\| \leq \eps$ then
    $E_{w_1,w_2}$ implies that this path
    has length at most $2n(45 e 2^n \log(1/p))^n$.
\end{lemma}
\begin{proof}%[Proof of Lemma \ref{lem:shadow-path-is-local-rv}]
	\label{proof:shadow-path-is-local-rv}
	Let us first assume that $\|w_1 - w_2\| \leq 2\eps/n$.
	Consider the points $v_1,v \in P(A\cap C(w_2,5\eps))$
	such that $\sprod{w_1}{v_1} \geq 1$ and $\sprod{w_2}{v} \geq \sprod{w_2}{v_1}$.
	By \Cref{lem:local-hardest-part}, assuming $E_{w_1,w_2}$, any such points have bounded norm.
	Hence, we can take $v_1$ to be a vertex maximizing $\sprod{w_1}{\cdot}$
	and $v \in P(A\cap C(w_2,5\eps))$ be any vertex on the shadow path from $w_1$ to $w_2$.
	
	Again by \Cref{lem:local-hardest-part}, assuming $E_{w_1,w_2}$,
	every $a \in A$ such that $\sprod{a}{v}=1$ satisfies $a \in C(w_2, (2+2/n)\eps)$,
	meaning that $v,v_1 \in P(A)$ as well.
	
	Now \Cref{lem:shadowpathcharacterization} implies that if $E_{w_1,w_2}$ then
	any vertex of $P(A \cap C(w_2,5\eps)$ on its shadow path from $w_1$ to $w_2$
	is a shadow vertex of $P(A)$ on the shadow path from $w_1$ to $w_2$.
	Hence the shadow path on $P(A)$ from $w_1$ to $w_2$ is a $(w_1,w_2)$-local random variable.
	
	The upper bound follows because
	every vertex on the shadow path is visited at most once
	and, assuming $E_{w_1,w_2}$, almost surely every vertex on the shadow path
	is induced by $n$ constraints out of $A \cap C(w_2, (2+2/n)\eps)$.
	The total number of subsets of size $n$ of $A \cap C(w_2, (2+2/n)\eps)$
	is at most $\abs{A \cap C(w_2, (2+2/n)\eps)}^n \leq (45 e 2^n \log(1/p))^n$
	by $E_{w_1,w_2}$.
	
	To extend the conclusion to the case when $2\eps/n < \|w_1 - w_2\| \leq \eps$,
	pick $w_1 = v_1, v_2, \dots, v_{2n+1} = w_2$ evenly spaced on the shortest
	geodesic segment connecting $w_1$ and $w_2$. For every $k \in [2n]$,
	by \Cref{lem:geodesic}
	the shadow path from $v_k$ to $v_{k+1}$ satisfies $\|v_k - v_{k+1}\| \leq 2\eps/n$
	and is thus a $(v_k, v_{k+1})$-local random variable and $E_{v_k, v_{k+1}}$
	implies that this shadow path has length at most $(45 e 2^n \log(1/p))^n$
	when $E_{v_k, v_{k+1}}$.
	
	Now observe that the shadow path from $w_1$ to $w_2$ is obtained by concatenating
	the shadow paths from $v_k$ to $v_{k+1}$ for $k \in [n]$.
	Since $E_{w_1,w_2}$ implies $E_{v_k,v_{k+1}}$ for every $k \in [2n]$,
	each of the shadow paths from $v_k$ to $v_{k+1}$ is a $(w_1,w_2)$-local random variable.
	Hence the shadow path from $w_1$ to $w_2$ is a $(w_1,w_2)$-local random variable
	and has length at most $2n(45 e 2^n \log(1/p))^n$.
\end{proof}

\begin{lemma}\label{lem:concentration}
    Let $0 < p < m^{-2n}$ and let
    $\eps = \eps(m,n,p) < 1/76$ be as in \Cref{lem:density}
    and let $k \geq 2\pi /\eps$ be the smallest number divisible by $76$.
    Let $W \subset \R^n$ be a fixed 2D linear subspace
    and let $w_1,\dots,w_k,w_{k+1}=w_1 \in W \cap \sfe$
    be equally spaced around the circle.
    Assume for every $i \in [k]$ that $K_i \geq 0$
    is a $(w_i,w_{i+1})$-local random variable
    and there exists $U \leq m^n$ such that $K_i \leq U$ whenever $E_{w_i,w_{i+1}}$.
    Furthermore assume that $\E[\sum_{i=1}^k K_i] \leq O(n^2 m^{\frac{1}{n-1}})$.
    Then
    \[
        \Pr\Big[\;
                \abs*{\sum_{i\in[k]} K_i
                    - \E\big[\sum_{i\in[k]} K_i \big]}
                \geq t_p\Big]
        \leq
        4p
    \]
    for $t_p = \max\left(\sqrt{O(U n^2m^{\frac{1}{n-1}} \log(1/p))}, O(U\log(1/p))\right)$.
\end{lemma}
\begin{proof}%[Proof of Lemma \ref{lem:concentration}]
	\label{proof:concentration}
	Let $F$ denote the event that $E_{v_1,v_2}$
	holds for every $v_1, v_2 \in \sfe$.
	By \Cref{lem:density} we have $\Pr[F] \geq 1-2p$.
	
	Define $E_i := E_{w_i,w_{i+1}}$, $i \in [k]$.
	Our first observation is that
	$\Pr[\sum_{i=1}^k K_i = \sum_{i=1}^k K_i 1[E_i]] \geq \Pr[F] \geq 1 - p$.
	Since both sums only take values in the interval $[0, km^n]$, it follows that
	\[
	\Big|\E[\sum_{i=1}^k K_i] - \E[\sum_{i=1}^k K_i1[E_i]]\Big| \leq 2km^n p\leq 1.
	\]
	From the above statements we deduce that
	\begin{align*}
		\Pr\left[\abs*{\sum_{i=1}^k K_i - \E\left[\sum_{i=1}^k K_i \right]} > t_p\right]
		&\leq \Pr\left[\abs*{\sum_{i=1}^k K_i - \E\left[\sum_{i=1}^k K_i1[E_i] \right]} > t_p-1\right] \\
		&\leq \Pr[\neg F] + \Pr\left[F \wedge \abs*{\sum_{i=1}^k K_i - \E\left[\sum_{i=1}^k K_i1[E_i] \right]} > t_p-1\right] \\
		&\leq 2p + \Pr\left[F \wedge \abs*{\sum_{i=1}^k K_i1[E_i] - \E\left[\sum_{i=1}^k K_i1[E_i] \right]} > t_p-1\right] \\
		&\leq 2p + \Pr\left[\abs*{\sum_{i=1}^k K_i1[E_i] - \E\left[\sum_{i=1}^k K_i1[E_i] \right]} > t_p-1\right]
	\end{align*}
	We will now upper bound the last term.

	For $j \in [76]$ define $I_j = \{i \in [k]\mid i \equiv j \mod 76\}$,
	forming a partition $I_1 \cup \dots\cup I_{76} = [k]$.
	Observe that $w_1,\dots,w_k$ are placed on a unit circle
	and every $[w_i, w_{i+1}]$ is an edge of $\conv(w_1,\dots,w_k)$.
	As such we know that $\sum_{i \in [k]} \|w_i - w_{i+1}\| \leq 2\pi$.
	Since $k \geq 2\pi /\eps$ that gives us $\|w_i - w_{i+1}\| \leq \eps$
	for every $i \in [k]$.
	Next, from $\eps \leq 1/76$ we know that $k \leq 2\pi /\eps + 76 \leq 8/\eps$.
	Since $k \geq 4$ we have $\sum_{i \in [k]} \|w_i - w_{i+1}\| \geq 4$
	and hence $\|w_i - w_{i+1}\| \geq 4/k \geq \eps/2$ for every $i \in [k]$.
	Last, we use that
	$\|w_i - w_{i+76}\| \leq \sum_{j=i}^{i+75}\|w_j - w_{j+1}\| \leq 76 \eps \leq 1$ to deduce
	\[
	\|w_i - w_{i + 76}\| \geq
	\frac{1}{\pi}\sum_{k=i}^{i+75}\|w_k - w_{k+1}\|
	\geq \frac{76}{\pi}\cdot \eps/2 > 12\eps.
	\]
	This lets us conclude that if $i, i' \in I_j$ are
	distinct then $\|w_i - w_{i'}\| > 12\eps$.
	In particular, for any $j \in [76]$ the random variables $K_i1[E_i]$
	for $i \in I_j$ are mutually independent since
	they are functions of $A$ intersected with disjoints subsets of $\sfe$
	due to being local random variables.

	For any $i \in [k]$, the random variable $K_i 1[E_i] \in [0, U]$
	has variance at most
	\[
	\E[K_i 1[E_i]] \cdot U
	\leq \frac{O(n^2m^{\frac{1}{n-1}})}{k} \cdot U
	\]
	by \Cref{lem:bhatia-davis}.
	
	We apply \Cref{lem:bernstein} to the random variables
	$K_i1[E_i]$ for $i \in [k]$ and obtain
	\[
	\Pr\left[\abs*{\sum_{i=1}^k K_i1[E_i] - \E\left[\sum_{i=1}^k K_i1[E_i] \right]} > t_p-1\right]
	\leq 2\exp\left(\frac{-8(t_p-1)^2}{1900(UO(n^2 m^{\frac{1}{n-1}}) + (t_p-1)U)}\right).
	\]
	By filling in $t_p$, we find that the right-hand side
	of the above inequality is at most $2p$.
	
	Putting the bounds together we get our desired inequality
	\[
	\Pr\Big[\;
	\sum_{i\in[k]} K_i
	\geq \E\big[\sum_{i\in[k]} K_i \big] + t\Big]
	\leq
	4p.
	\]
\end{proof}

\subsection{Concentration of the shadow size around its mean}
\label{sec:shadow}
To illustrate the use of the above technical result,
we show in this subsection that $\abs{\mathcal S(P(A),W)}$
is concentrated around its mean when $m > 2^{O(n^3)}$.

Recall that by \Cref{thm:borgwardt} we have
$\E[\abs{\mathcal S(P(A),W)}] = \Theta(n^2m^{\frac{1}{n-1}})$.
\ubshadow*
\begin{proof}
	From \Cref{cor:caps}, we know that $\eps^{n-1} \leq \frac{1}{76^{n-1}}$.
	As such, the lower bound on $p$ implies that $\eps(m,n,p) < 1/76$.
	
	Let $w_1,\dots,w_k$ be as in \Cref{lem:concentration} and let $K_i$ denote
	the number of edges on the shadow path from $w_i$ to $w_{i+1}$.
	By \Cref{lem:shadow-path-is-local-rv}, each $K_i$ is a $(w_i, w_{i+1})$-local random variable
	which satisfies $K_i \leq 2n(45 e 2^n \log(1/p))^n$ when $E_{w_i, w_{i+1}}$.
	
	By \Cref{thm:borgwardt} we get
	$\sum_{i\in [k]} K_i = \abs{\mathcal S(P(A),W)}$ almost surely,
	hence $\E[\sum_{i\in[k]}K_i] \leq O(n^2 m^{\frac{1}{n-1}})$.
	We apply \Cref{lem:concentration} to $\sum_{i\in[k]} K_i$:
	\begin{align*}
		\Pr\big[\Big|\abs{\mathcal S(P(A),W)} - \E[\abs{\mathcal S(P(A),W)}]\Big| > t\big]
		= \Pr\big[\Big| \sum_{i\in[k]}K_i - \E[\sum_{i\in[k]}K_i] \Big|> t\big]
		\leq 4p.
	\end{align*}
\end{proof}

\subsection{Upper bound on the diameter}
In this section we prove our high probability upper bound on
$\diam(P(A))$. We start by proving that for fixed $W$
the vertices in $\mathcal S(P(A),W)$ are connected by short paths,
where we aim for an error term smaller than that of \Cref{thm:shadowbound}.
We require the following abstract diameter bound from \cite{EHRR10}.
We will only need the Barnette--Larman style bound.

\begin{theorem}\label{thm:BL}
    Let $G=(V,E)$ be a connected graph, where the vertices $V$ of $G$
    are subsets of $\{1,\dots,k\}$ of cardinality $n$ and the edges $E$ of $G$
    are such that for each $u,v\in V$ there exists a path connecting $u$ and $v$
    whose intermediate vertices all contain $u \cap v$.

    Then the following upper bounds on the diameter of $G$ hold:
    \[2^{n-1} \cdot k - 1\text{ (Barnette--Larman)}, \hspace{4em}
k^{1+\log n} - 1 \text{ (Kalai--Kleitman)}.\]
\end{theorem}

To confirm that the above theorem indeed gives variants of the Barnette--Larman
and Kalai--Kleitman bounds,
let $A=\{a_1,...,a_m\} \subset \sfe$ be in general position.
For a vertex $x \in P(A)$, we denote $A_x = \{a \in A : \sprod a x = 1\}$.
Consider the following sets
\begin{align*}
V &= \{ A_x : \text{$x$ is a vertex of $P(A)$} \},\\
E &= \{ \{A_x, A_y\} : \text{$[x,y]$ is an edge of $P(A)$} \}.
\end{align*}
One can check that $G=(V,E)$ satisfies almost surely the assumptions of \cref{thm:BL} which
therefore shows that the combinatorial diameter of $P(A)$ is less than
$\min(2^{n-1}\cdot m-1, m^{1+\log n} - 1)$.
Up to a constant factor difference, these bounds correspond to the same bounds described in
the introduction.

Now we use the Barnette--Larman style bound to bound the length of the local paths.
\begin{lemma}\label{lem:shortest-path-is-local-rv}
    Let $w_1, w_2 \in \sfe$ satisfy $\|w_1-w_2\| \leq \eps$, where $\eps = \eps(m,n,p)$
    is as in \Cref{lem:density}.
    Furthermore, let $K$ denote the maximum over all $w \in [w_1, w_2]$
    of the length of the shortest path from
    a maximizer $v_w\in P(A)$ of $\sprod {w}\cdot$ to the maximizer of
    $\sprod {w_2}\cdot$ of which every vertex $v \in P(A)$ on
    the path satisfies $\sprod {w_2} v \geq \sprod {w_2}{v_w}$.
    Then $K$ is a $(w_1,w_2)$-local random variable and
    $E_{w_1,w_2}$ implies that $K_i$ is at most $45e n 4^n \log(1/p)$.
\end{lemma}
\begin{proof}%[Proof of Lemma \ref{lem:shortest-path-is-local-rv}]
	\label{proof:shortest-path-is-local-rv}
	We start by assuming $\|w_1 - w_2\| \leq 2\eps/n$.
	Let $w \in [w_1,w_2]$ and let $v_w \in P(A)$
	be a vertex maximizing $\sprod w \cdot$.
	By \Cref{lem:local-hardest-part}, assuming $E_{w_1,w_2}$, for every vertex $v\in P(A)$
	satisfying $\sprod{w_2} v \geq \sprod{w_2}{v_w}$ and every $a \in A$
	such that $\sprod a v \geq 1$ we have $a \in A \cap C(w_2, (2+2/n)\eps)$.
	
	First, this implies that if $E_{w_1,w_2}$ and if $v \in \R^n$
	is satisfies $\sprod{w_2}v \geq \sprod{w_2}{v_w}$
	then we need only inspect $A \cap C(w_2, (2+2/n)\eps)$
	to decide if $v$ is a vertex of $P(A)$.
	From this we conclude that if $E_{w_1, w_2}$
	then the shortest path described in the lemma statement
	can be computed knowing only $A \cap C(w_2,(2+2/n)\eps)$.
	This implies that the path length is a $(w_1,w_2)$-local
	random variable.
	
	Second, assuming $E_{w_1,w_2}$ we consider the sets
	\begin{align*}
		\widehat{V}
		&= \{ v \in P(A) : v \text{ is a vertex and } \sprod{w_2}{v} \geq \sprod{w_2}{v_1} \} , \\
		\widehat{A}
		&= \{ a\in A : \text{ there exist a vertex $v \in \widehat{V}$ such that $\sprod{a}{v}=1$} \}
		\subset A \cap C\left({w_2}, \left(2 + \frac{2}{n}\right)\eps\right).
	\end{align*}
	The last inclusion follows directly from \Cref{lem:local-hardest-part}.
	
	Recall the notation $A_v = \{a \in A : \sprod a v = 1\}$ for vertices $v \in P(A)$.
	We will apply \Cref{thm:BL} to the graph
	\begin{align*}
		V &= \{ A_v : v\in \widehat{V} \} \simeq \widehat{V},\\
		E &= \{ \{A_{v_1}, A_{v_2}\} :  v_1 ,v_2 \in \widehat{V} ,\ \text{$[v_1,v_2]$ is an edge of $P(A)$} \}.
	\end{align*}
	We need to check that the assumptions of \Cref{thm:BL} are met.
	First we note that almost surely $P(A)$ is a simple polytope and thus the
	vertices of the graph $(V,E)$ are subsets of $A$ of cardinality $n$.
	Consider two vertices $A_v=\{a_{i_1},\ldots,a_{i_n}\},
	A_{v'}=\{a_{i'_1},\ldots,a_{i'_n}\} \in V$.
	Observe that the set
	\[ F = \{ x \in P(A) : \sprod{x}{a}=1 \ \forall a \in A_v\cap A_{v'} \} \]
	is the minimum face of $P(A)$ containing both $v$ and $v'$.
	We build paths $v_0 = v, v_1 , \ldots , v_k $ and $v'_0 = v', v'_1 , \ldots , v'_{k'} $ satisfying the following monotonicity properties
	\begin{align*}
		\sprod {w_2} {v}
		&= \sprod {w_2} {v_0}
		\leq \sprod {w_2} {v_1}
		\leq \cdots
		\leq \sprod {w_2} {v_k}
		= \argmax\{\sprod {w_2} x : x\in F\} , \\
		\sprod {w_2} {v'}
		&= \sprod {w_2} {v'_0}
		\leq \sprod {w_2} {v'_1}
		\leq \cdots
		\leq \sprod {w_2} {v'_{k'}}
		= \argmax\{\sprod {w_2} x : x\in F\} .
	\end{align*}
	Moreover one can assume that $v_k = v'_{k'}$ by potentially completing the
	paths moving along the edges of $\argmax\{\sprod {w_2} x : x\in F\}$ (in
	the case this face contains more than one vertex).  By construction all
	vertices $v_i$ and $v'_i$ belong to $\widehat{V}$.
	Stitching the two paths and adopting the dual point of view we found a path
	$A_{v} = A_{v_0} , \ldots , A_{v_k} = A_{v'_{k'}} , \ldots A_{v'_0} =
	A_{v'}$ whose vertices contain the intersection $A_v\cap A_{v'}$.
	
	We can thus apply \Cref{thm:BL} and
	conclude that there is a path in the graph $(V,E)$ from $A_{v_1}$ to $A_{v_2}$
	of length at most $2^{n-1} \cdot \abs{A \cap C(w_2,(2+2/n)\eps)}$.
	It follows that $K \leq 2^{n-1}\cdot\abs{A \cap C(w_2,(2+2/n)\eps)}$.
	
	To extend the conclusion to the case when $2\eps/n < \|w_1 - w_2\| \leq \eps$,
	we do the same as in the proof of \Cref{lem:shadow-path-is-local-rv}.
\end{proof}

\begin{theorem}\label{thm:shadow-plane-diam}
    Let $0 < p < m^{-2n}$ and let
    \[t_p = \max\left(\sqrt{O(U n^2  m^{\frac{1}{n-1}} \log(1/p))}, O(U \log(1/p))\right)\]
    for $U = O(n 4^{n} \log(1/p))$.
    If $W\subset\R^n$ is
    a fixed 2D linear subspace and $A \sim \Pois(\sfe,m)$, the largest distance $T$
    between any two shadow vertices satisfies
    \[
        \Pr[T \geq O(n^2 m^{\frac{1}{n-1}}) + t_p ] \leq 4p
    \]
\end{theorem}
\begin{proof}
    Let $w_1,\dots,w_k$ be as in \Cref{lem:concentration} and let $K_i$ denote
    the maximum over all $w \in [w_i, w_{i+1}]$
    of the length of the shortest path from a shadow vertex $v_w$ maximizing
    $\sprod w \cdot $ to a vertex maximizing $\sprod {w_{i+1}} \cdot$
    such that every vertex $v$ on this path satisfies
    $\sprod {w_{i+1}} v \geq \sprod {w_{i+1}}{v_w}$.
    From \Cref{lem:shortest-path-is-local-rv} we know that $K_i$ is
    a $(w_i,w_{i+1})$-local random variable
    and $K_i \leq 45 e n 4^n \log(1/p)$ whenever $E_{w_i,w_{i+1}}$.
    Now recall \Cref{thm:borgwardt}.
    Observe that $T \leq \sum_{i \in [k]}K_i$ almost surely by concatenating the
    above-mentioned paths, and note that
    that $\sum_{i \in [k]} K_i \leq \mathcal S(P(A),W)$ holds almost surely,
    which implies $\E[\sum_{i\in[k]} K_i] = O(n^2 m^{\frac{1}{n-1}})$.
    We apply \Cref{lem:concentration} to $\sum_{i\in[k]} K_i$
    and get the desired result.
\end{proof}

\ubpa*

\section{Lower Bounding the Diameter of \texorpdfstring{$P(A)$}{P(A)}}
\label{sec:lower-bound}
To begin, we first reduce to lower bounding the diameter of the polar
polytope $P^\circ$, corresponding to a convex hull of $m$ uniform points on
$\sfe$, via the following simple lemma. 

\begin{restatable}[Diameter Relation]{lemma}{diamrel}
\label{lem:rel-diam}
For $n \geq 2$, let $P \subseteq \R^n$ be a simple bounded polytope
containing the origin in its interior and let $Q = P^\circ := \{x \in \R^n:
\langle x, y \rangle \leq 1, \forall y \in P\}$ denote the polar of $P$.
Then, $\diam(P) \geq (n-1)(\diam(Q)-2)$. 
\end{restatable}
\begin{proof}%[Proof of Lemma~\ref{lem:rel-diam}]
	\label{proof:rel-diam}
	If $\diam(Q) \leq 1$, the statement is trivial, so we may assume that
	$\diam(Q) \geq 2$. Let $a_1,a_2 \in Q$ be vertices of $Q$ at distance
	$\diam(Q) \geq 2$. Since $P$ is bounded, note that $0$ is in the interior of
	$Q$ and hence $a_1,a_2 \neq 0$. We must show that there exists a path from
	$a_1$ to $a_2$ of length $L \geq 2$ such $\diam(P) \geq (n-1)(L-2)$.  
	
	Let $F_i := \{x \in P: \langle a_i, x \rangle = 1\}$, $i \in [2]$, the
	corresponding facets of $P$. Pick the two vertices $v_1 \in F_1$, $v_2 \in
	F_2$ whose distance in $P$ is minimized. Let $v_1 := w_0,\dots,w_D := v_2$ be
	a shortest path from $v_1$ to $v_2$ in $P$. Here $w_0,\dots,w_D$ are all
	vertices of $P$, and $[w_i,w_{i+1}]$, $0 \leq i \leq D-1$, are edges of $P$.
	By definition, $D \leq \diam(P)$. 
	
	To complete the proof, we will extract a walk from $a_1$ to $a_2$ in $Q$ from
	the path $w_0,\dots,w_D$ of length at most $D/(n-1)+2$. For this purpose,
	let $Q_i := Q \cap \{x \in \R^n: \langle x, w_i \rangle = 1\}$, $0 \leq i
	\leq D$, denote the facet of $Q$ induced by $w_i$. By our assumption that $P$
	is simple, each $Q_i$, $i \in [D]$, is a $(n-1)$-dimensional simplex, and
	hence there exists $S_i \subseteq \verts(Q)$, $|S_i| = n$, such that $Q_i :=
	\conv(a: a \in S_i)$. In particular, the combinatorial diameter of each
	$Q_i$, $0 \leq i \leq D$, is $1$. That is, every distinct pair of vertices of
	$Q_i$ induces an edge of $Q_i$, and hence an edge of $Q$. 
	
	By the above discussion, note that if $a_1,a_2 \in S_0$, then $a_1,a_2$ are
	adjacent in $Q$. Since we assume that the distance between $a_1,a_2$ is at
	least $2$, we conclude that $a_1,a_2 \notin S_0$, and hence that $D \geq 1$.
	Furthermore, since we assume that $v_1,v_2$ are at minimum distance in $P$
	subject to $v_1 \in F_1, v_2 \in F_2$, we conclude that $a_1 \in S_0
	\setminus \cup_{j=1}^D S_j$ and $a_2 \in S_L \setminus \cup_{j=0}^{D-1} S_j$,
	since otherwise we could shortcut the path.  
	
	We now define a walk $a_1 = u_0,\dots,u_L = a_2$, for some $L \geq 2$, from
	$a_1$ to $a_2$ in $Q$ as follows. Letting $l_0 = 0$ and $S_{D+1} :=
	\emptyset$, for $i \geq 1$ inductively define $l_i := \max \{j \geq l_{i-1}:
	\cap_{r=l_{i-1}}^j S_r \neq \emptyset\}$ and let $L = \min \{i \geq 1: l_i =
	D\}+1$. For $1 \leq i \leq L-1$, choose $u_i$ from $\cap_{r=l_{i-1}}^{l_i}
	S_r$ arbitrarily. To relate the length of the walk to $D$, we will need the
	following claim.
	
	\begin{claim} For any interval $I \subseteq \{0,\dots,D\}$, $|\cap_{i \in I}
		S_i| \geq n-|I|+1$.
	\end{claim}
	\begin{proof}
		First note that $|S_j \cap S_{j+1}| = n-1 = |S_j|-1$, $0 \leq j \leq D-1$,
		since $P$ is simple and $S_j \cap S_{j+1}$ indexes the tight constraints of
		an edge of $P$. In particular, $|S_j \setminus S_{j+1}| = 1$, $0 \leq j \leq 
		D-1$. Thus, for an interval $I = \{c,c+1,\dots,d\} \subseteq \{0,\dots,D\}$,
		we see that $|\cap_{i=c}^d S_i| \geq |\cap_{i=c}^{d-1} S_i| - |S_{d-1}
		\setminus S_d| = |\cap_{i=c}^{d-1} S_i|-1 \geq |S_c|-(d-c) = n+1-|I|$.
	\end{proof}
	
	Applying the claim to the interval $I=\{l_{i-1},\dots,l_i+1\}$, $1 \leq i
	\leq L-1$, we see that $\cap_{r=l_{i-1}}^{l_i+1} S_r = \emptyset$ implies
	that either $l_i = D$ or that $|I| \geq n+1 \Leftrightarrow l_{i}-l_{i-1}
	\geq n-1$. In particular, $l_i-l_{i-1} \geq n-1$ for $0 \leq i \leq L-2$ and
	$l_{L-1}-l_{L-2} \geq 1$ (since $l_{L-1}=D$ and $l_{L-2} < D$).
	
	Let us now verify that $a_1 = u_0,u_1,\dots,u_L = a_2$ induces a walk in $Q$.
	Here, we must check that $[u_i,u_{i+1}]$, $0 \leq i \leq L-1$, is an edge of
	$Q$. By construction $u_i,u_{i+1}$ are both vertices of the simplex
	$Q_{l_i}$. Furthermore, $u_i \neq u_{i+1}$, since either $u_i = a_1 \neq
	u_{i+1}$ or $u_{i+1} = a_2 \neq u_i$ or $u_{i+1} \in S_{l_i+1}$ and $u_i
	\notin S_{l_i+1}$. Thus, $[u_i,u_{i+1}]$ is indeed an edge of $Q_i$ and thus
	of $Q$, as explained previously. Note by our assumption that $a_1$ and $a_2$, we indeed have $2 \leq \diam(Q) \leq L$.
	
	We can now compare the diameters of $P$ and $Q$ as follows:
	\[
	\diam(P) \geq D = l_{L-1}-l_0 = \sum_{i=1}^{L-1} (l_i-l_{i-1})  
	\geq \sum_{i=1}^{L-2} (n-1) = (n-1)(L-2) \geq (n-1)(\diam(Q)-2), \]
	as needed.
\end{proof}

We then
associate any ``antipodal'' path to a continuous curve on the sphere
corresponding to objectives maximized by vertices along the path. From here,
we decompose any such curve into $\Omega(m^{\frac{1}{n-1}})$ segments whose
endpoints are at distance $\Theta(m^{-1/(n-1)})$ on the sphere. Finally, we
apply a suitable union bound, to show that for any such curve, an $\Omega(1)$
fraction of the segments induce at least $1$ edge on the corresponding path.

Building on Lemma \ref{lem:rel-diam}, we turn to proving the lower bound for $Q(A)$.

% Defining a (sub)sequence for a given objective function
For a discrete set $N\subset S^{n-1}$, a point $x_0\in N$ and a positive number $\eps>0$ we denote by
\[X_k := X_k(N,x_0,\eps) = \{ \mathbf{x} \in N^k : x_i \neq x_j \text{ and } 6 \eps \leq \norm{x_i-x_{i+1}} \leq 8 \eps \text{ for any } 0\leq i < j \leq k \} \]
the set of all sequences of $k$ distinct points in $N$ with jumps of length between $6\eps$ and $8\eps$ (including an extra initial jump between $x_0$ and $x_1$).
% Counting the number of (sub)sequences
\begin{lemma} \label{lem:cardXk}
    Let $\eps>0$. If $N \subset S^{n-1}$ is a maximal $\eps$-separated set, then
    \[|X_k| \leq  (17^{n-1})^k. \]
\end{lemma}
Note that a maximal $\varepsilon$-separated set is also an $\varepsilon$-net.
\begin{proof}[Proof of Lemma \ref{lem:cardXk}]
	\label{proof:cardXk}
	For any $x\in N$ we find an upper bound for the number of points $y\in N$ such that $6\varepsilon \leq \|x-y\|\leq 8 \varepsilon$. Recall that $C(x,r)$ denotes the closed spherical cap centered at $x$ with radius $r>0$. Since $N$ is $\varepsilon$-separated, for any different points $y_1,y_2\in N$ we have
	\[\textrm{int}(C(y_1,\varepsilon/2)) \cap C(y_2,\varepsilon/2) = \emptyset.\]
	Taking a union of spherical caps centered at all points inside the annulus, we obtain a subset of the inflated annulus
	\[C(x,17\varepsilon/2) \setminus \textrm{int}(C(x,11\varepsilon/2)).\]
	Since the caps $C(y,\eps/2)$, $y\in N$, have pairwise disjoint interiors, the volume of their union is the sum of the volumes.
	Hence, the maximal number of points in the annulus is bounded by
	\[\frac{\sigma(C(x,17\varepsilon/2))-\sigma(C(x,11\varepsilon/2))}{\sigma(C(x,\varepsilon/2))} \leq \frac{\sigma(C(x,17\varepsilon/2))}{\sigma(C(x,\varepsilon/2))}.\]
	Using Lemma \ref{lem:cap-area} we have
	\[|\{y\in N:\; 6\varepsilon \leq \|x-y\|\leq 8 \varepsilon\}| \leq \frac{(17/2)^{n-1}}{(1/2)^{n-1}} = 17^{n-1}.\]
	Thus, the overall number of paths in $X_k$ is bounded by
	\[|X_k| \leq 17^{k(n-1)}.\]
\end{proof}
% A continuous objective map generates a (sub)sequence
\begin{lemma} \label{lem:subseq}
    Let $f\colon [0,1] \to S^{n-1}$ be a continuous function.
    Let $\eps>0$ and $N\subset S^{n-1}$ be an $\eps$-net, such that $f(0)\in N$.
    There exist $k\in \N_0$, $0 \leq t_0<t_1<\cdots < t_k \leq 1$ and $x_0,\ldots,x_k \in N$ such that
    \begin{enumerate}
        % \item $\norm{f(0)-x_0} \leq \eps$,
        \item $\norm{f(t_i) - x_i} \leq \eps $ for any $i\in \{0,\ldots,k\}$, 
        \item $\norm{f(t) - x_i} \geq \eps $ for any $i\in \{0,\ldots,k\}$ and $t>t_i$, 
        \item $(x_1,\ldots,x_k)\in X_k(N,x_0,\eps)$,
        \item $\norm{x_k - f(1)} < 7 \eps$.
    \end{enumerate}
\end{lemma}
\begin{proof}%[Proof of Lemma \ref{lem:subseq}]
	\label{proof:subseq}
	We build the desired couple of sequences $(x_i)$ and $(t_i)$ by induction. We start by taking $x_0=f(0)$ and
	\[t_0=\sup\{t\geq 0 : \|f(t)-x_0\|\leq\varepsilon\}.\]
	Note that with these choices, we have a couple of (very short) sequences for which 1-3 are fulfilled.
	
	Assume that $x_0,\ldots,x_{\ell}$ and $0\leq t_0 < \ldots < t_{\ell}\leq 1$ are sequences for which 1-3 hold true. 
	
	If $\|x_{\ell}-f(1)\|<7\varepsilon$ then we may take $k=\ell$, and we are done. 
	
	Assume otherwise, and define
	\[t' = \min \{ t \in [t_{\ell},1] : \exists x_{\ell+1} \in N \text{ with } \norm{f(t) - x_{\ell+1}} \leq \eps \text{ and } \norm{x_{\ell+1}-x_{\ell}} \geq 6 \eps \}, \]
	Since 4 is not fulfilled, the set in non-empty (it contains $1$) and $t'$ is well defined. We take $x_{\ell+1}$ as it appears in the definition of $t'$. Set
	\[t_{\ell+1} = \sup\{t\in[0,1] : \|f(t)-x_{\ell+1}\|\leq \varepsilon\}.\]
	By 2 for any $i\leq \ell$
	\[\|f(t_{\ell+1}) - x_i\| >\varepsilon,\]
	hence $x_i\ne x_{\ell+1}$. Combining this with the definition of $t_{\ell+1}$ and $x_{\ell+1}$ we only need to show that $\|x_{\ell}-x_{\ell+1}\|\leq 8\varepsilon$ in order to get that $0\leq t_0 <\ldots <t_{\ell} < t_{\ell+1} \leq 1$ and $x_0,\ldots,x_{\ell+1}$ fulfill 1-3.
	
	By the minimality of $t'$, for any $s\in(t_{\ell},t')$ we have $\|x_{\ell}-f(s)\|\leq 7\varepsilon$, otherwise there would be $x'\in N$ such that $\|x'-f(s)\|\leq \varepsilon$ but $\|x_{\ell}-x'\|\geq 6\varepsilon$, hence $t'\leq s$ in contradiction to the definition of $s$. Hence	
	\[\|x_{\ell}-x_{\ell+1}\| \leq \|x_{\ell}-f(s)\| + \|f(s) - f(t')\| + \|f(t')-x_{\ell+1}\| \leq 7\varepsilon + \|f(s) - f(t')\| + \varepsilon.\]
	
	This holds for all $s\in(t_{\ell},t')$. By continuity of $f$ we may take $s\nearrow t'$ and have $\|f(s) - f(t')\|\to 0$. Thus $\|x_{\ell}-x_{\ell+1}\|\leq 8\varepsilon$.
	
	Since $N$ is finite and the points $x_0,\ldots,x_{\ell}$ are distinct the process must end at most after $|N|$ steps.
\end{proof}

% from a path to an objective function
\begin{lemma}\label{lem:pathobjective}
    Let $A \subset S^{n-1}$ be a finite subset of the sphere.
    Let $[a_0,a_1]$, $[a_1,a_2]$, \ldots, $[a_{\ell-1},a_\ell]$ be a path along the edges of $Q(A)$.
    There exists a continuous function $f\colon [0,1] \to S^{n-1}$ and $0=s_0 < s_1 < \cdots < s_{\ell+1} =1$ such that $f(0)=a_0$, $f(1)=a_\ell$, and for any $i \in \{ 0,1,\ldots,\ell \}$ and any $t \in [s_i,s_{i+1}]$,
    \[ a_i \in \argmin_{a\in A} (\norm{f(t)-a}) .\] 
\end{lemma}
\begin{proof}%[Proof of Lemma \ref{lem:pathobjective}]
	\label{proof:pathobjective}
	First we consider the case where the path consist of a single edge, i.e. $\ell=1$.
	Consider a point $x\in S^{n-1}$ and a real $r>0$ such that the cap $C(x,r)$ contains $a_0$ and $a_1$ on its boundary and no point of $A$ in its interior.
	A possible choice is given by the circumscribed cap of any facet of $Q(A)$ which contains $[a_0,a_1]$ as an edge.
	Now we set $f$ such that it interpolates $a_0$, $x$ and $a_1$ by two geodesic segments,
	\begin{align*}
		f(t) &= \frac{\tilde{f}(t)}{\norm{\tilde{f}(t)}}, &
		\tilde{f}(t) =
		\begin{cases}
			(1-2t) a_0 + 2t x , & t\in[0,\frac{1}{2}],\\
			(2-2t) x + (2t-1) a_1, & t\in[\frac{1}{2},1].
		\end{cases} 
	\end{align*}
	By construction we get that for any $t\in [0,\frac{1}{2}]$ (resp. $t\in [\frac{1}{2},1]$), the cap $C(f(t), \norm{f(t)-a_0})$ (resp. $C(f(t), \norm{f(t)-a_1})$) is a subset of $C(x,r)$.
	Thus it contains $a_0$ (resp. $a_1$) on its boundary and no point of $A$ in its interior.
	This implies that $f(0)=a_0$, $f(1)=a_1$, and
	\begin{align*}
		a_0 &\in \argmin_{a \in A} (\norm{f(t)-a}) , \quad t \in [0,\frac{1}{2}] ,\\
		a_1 &\in \argmin_{a \in A} (\norm{f(t)-a}) , \quad t \in [\frac{1}{2},1] .
	\end{align*}
	This yields the proof in the case $\ell=1$ (with $s_0=0 < s_1 = \frac{1}{2} < s_{1+1} =1$).
	The general case follows by concatenating and renormalizing the functions corresponding to each edge.
\end{proof}

% deterministic lower bound
\begin{lemma} \label{lem:deterministicLB}
    Let $A\subset \sfe$ be a finite subset of the sphere, containing two points $a_+,a_- \in A$ such that $\norm{a_+-a_-} \geq 1$.
    Let $\eps>0$ and $N \subset \sfe$ be a maximal $\eps$-separated set, such that $a_+\in N$.
    Set $x_0=a_+$ and $k_0 = \lceil 1 / 8 \eps \rceil - 1$.
    It holds that
    \[\diam(Q(A)) 
    \geq \min_{k\geq k_0} \min_{\mathbf{x}\in X_k(N,x_0,\eps)}  \sum_{0\leq i \leq k-1} 1[ C(x_i, \eps/2) \cap A \neq \emptyset ] 1[ C(x_{i+1}, \eps/2) \cap A \neq \emptyset ] .\]
\end{lemma}
\begin{proof}%[Proof of Lemma \ref{lem:deterministicLB}]
	\label{proof:deterministicLB}
	The diameter of $Q(A)$ is at least the combinatorial distance between $a_+$ and $a_-$, i.e., the minimal number of edges required to form a path between these two vertices.
	Note that this minimum is realized for a path without loops.
	Let $[a_0,a_1]$, $[a_1,a_2]$, \ldots, $[a_{\ell-1},a_\ell]$ be such a path.
	Here we denote $a_0=a_+=x_0$ and $a_\ell =a_-$.
	
	Consider a function $f$ and a sequence $0=s_0 < s_1 < \cdots < s_{\ell+1} =1$ as in Lemma \ref{lem:pathobjective}, and consider $k\in \N_0$, $0\leq t_0<t_1<\cdots < t_k \leq 1$ and $x_0,\ldots,x_k \in N$ as in Lemma \ref{lem:subseq}.
	We set $j(0)\leq j(1) \leq \cdots \leq j(k)$ such that $t_i \in [s_{j(i)}, s_{j(i)+1}]$.
	In particular, with this notation set up we have
	\begin{align} \label{eq:0819c}
		\norm{x_i - x_{i+1}} 
		&\geq 6 \eps, 
		& i\in \{ 0,\ldots, k-1\},
	\end{align}
	\begin{align} \label{eq:0819b}
		\norm{a_{j(i)}-f(t_i)} 
		&= \min_{a\in A} \norm{a-f(t_i)},
		& i\in \{ 0,\ldots, k\},
	\end{align}
	and
	\begin{align} \label{eq:0819a}
		\norm{x_i-f(t_i)} 
		&%= \min_{x\in N} \norm{x-f(t_i)}
		\leq \eps,
		& i\in \{ 0,\ldots, k\}.
	\end{align}
	%where the last inequality holds because $N$ is an $\eps$-net.
	
	From \eqref{eq:0819a} we get
	\( C(f(t_i),3\eps/2) 
	\supset C(x_i,\eps/2) \).
	Hence, if $C(x_i,\eps/2) \cap A \neq \emptyset$, we have  that $\norm{a_{j(i)}-f(t_i)} \leq 3 \eps/2$ because of \eqref{eq:0819b}.
	Therefore if, for some $i\in \{0,\ldots , k-1\}$, both caps $C(x_i,\eps/2)$ and $C(x_{i+1},\eps/2)$ contain points of $A$, then 
	\begin{align*}
		&\norm{ a_{j(i)} - a_{j(i+1)} }
		\\&\geq \norm{ x_{i} - x_{i+1} } - \norm{ x_{i} - f(t_i) } - \norm{ f(t_i) - a_{j(i)} } - \norm{ a_{j(i+1)} - f(t_{i+1}) } -\norm{f(t_{i+1}) - x_{i+1}}
		\\&\geq 6 \eps - \eps - 3\eps/2 - 3\eps/2 - \eps
		= \eps 
		> 0 
	\end{align*}
	and we get $a_{j(i+1)} \neq a_{j(i)}$ which implies that $j(i) < j(i')$ for any $i'>i$.
	This shows that if
	\[ i,i' 
	\in I 
	= \{ i : C(x_i,\eps/2) \cap A \neq 0 \text{ and } C(x_{i+1},\eps/2) \cap A \neq 0 \} 
	\subset \{ 0,1,\ldots,k-1 \},\]
	with $i\neq i'$, 
	then $a_{j(i)}$ and $a_{j(i')}$ are distinct vertices of the path.
	Therefore 
	$$\ell \geq |I| = \sum_{0\leq i \leq k-1} 1[ C(x_i, \eps/2) \cap A \neq \emptyset ] 1[ C(x_{i+1}, \eps/2) \cap A \neq \emptyset ] .$$
	
	Also, we note that from
	\begin{align*}
		\norm{a_+-a_-}
		&\leq \norm{a_+ - x_0} + \sum_{1\leq i \leq k} \norm{ x_i - x_{i-1}} + \norm{x_k-a_-}
		% strict inequality in lemma 33 point 4.
		\\&<  \eps + k \times 8 \eps + 7 \eps
		= 8 (k+1) \eps
	\end{align*}
	we have $k\geq k_0$, and therefore
	\[ (x_0,\ldots,x_k) \in \cup_{k \geq k_0} X_k(N,x_0,\eps) . \]
\end{proof}

% Thm lower bound 
\begin{restatable}[Lower Bound for Q(A)]{theorem}{lbqa}
	\label{thm:LBlargem}
	There exist positive constants $c_2<1$ and $c_3>1$ independent of $n \geq 3$ and $m$ such that the following holds. Let $A = \{a_1,\dots,a_M\} \in \sfe$, where $M$ is Poisson with $\E[M] = m$,
	and $a_1,\dots,a_M$ are uniformly and independently distributed in $\sfe$.
	Then, with probability at least $1- e^{- c_3^{n-1} m^{1/(n-1)}} $, the combinatorial diameter of $Q(A)$ is at least $c_2 m^{1/(n-1)}$.
\end{restatable}

%\Comment[GB]{The statement of the theorem does not require $m$ to be sufficiently large, but it provides non trivial bounds only if $m \geq (1/c_2)^{n-1}$.}
%\Comment[GB]{I think that we could provide explicit (non optimal) values for the constants $c_2$ and $c_3$. That was actually my original goal while writing the proof. I got a bit lost in some details and decided to finish the proof with non explicit value.
%If we are happy with non explicit values, we can probably simplify some steps in the proof.
%And if we want explicit value, we need to develop more certain computation.}
\begin{proof}
    Without loss of generality $m\geq (1/c_2)^{n-1}$ since otherwise the statement of the theorem is trivial.

    In this proof the constants $1< c_3 < c_4 < c_5 < c_6 < c_2^{-1}$ are large enough constants, independent from $n$ and $m$.
    
    We set $\eps = c_6 m^{-1/(n-1)}$, and want to apply Lemma \ref{lem:deterministicLB}.
    Let $N$ be an $\eps$-net, obtained from a maximal $\varepsilon$-separated set, such that it contains a point $a_+$ from the set $A$.
    For independence properties needed later we take $a_+$ randomly and uniformly from the set $A$.
    With probability $1-e^{-m/2}$ we have that $A$ intersects the halfsphere $\{ u\in \sfe : \sprod{a_+}{u} \leq 0 \}$.
    In which case there exists a point $a_-\in A$ such that $\norm{a_+-a_-} \geq \sqrt{2} \geq 1$.
    Therefore we can apply Lemma \ref{lem:deterministicLB} with $x_0=a_+$.
    Combined with the union bound, we get
    \begin{align*}
        \Pr\left( \diam Q(A) \leq c_2 m^{1/(n-1)} \right)
        & \leq e^{-m/2} + \sum_{k\geq k_0} \sum_{\mathbf{x}\in X_k(N,x_0,\eps)}  \Pr\left(\sum_{0\leq i \leq k-1} B_i \leq c_2 m^{1/(n-1)} \right) ,
    \end{align*}
    where
    \[ k_0 
    = \lceil 1/8\eps \rceil + 1 
    \geq 1/8\eps 
    = m^{1/(n-1)} / 8 c_6 , \]
    and the summands in the probability are Bernoulli random variables
    \[ B_i = 1[ C(x_i, \eps/2) \cap A \neq \emptyset ] 1[ C(x_{i+1}, \eps/2) \cap A \neq \emptyset ] . \]
    For $1\leq i \leq k-1$, they are identically distributed, with failure probability
    \begin{align*}
        \Pr(B_i=0) 
        &\leq 2\Pr ( C(x_i,\eps/2) \cap A = 0 )
        = 2\exp\left( - m \sigma(C(x_i, \eps/2)) \right)  
        \\&\leq 2\exp\left( - m \left(\eps/4\right)^{n-1} \right)
        = 2\exp\left( - \left(\frac{c_6}{4}\right)^{n-1} \right)
        =: 1 - p .
    \end{align*}
    Note that we used \Cref{lem:cap-area} to lower bound the cap's volume
    $\sigma(C(x_i, \eps/2)) \geq (\eps/4)^{n-1} \sigma(C(x_i, 2))$.
    Since $N$ forms a maximal $\eps$-separated set and the $x_i$ are distinct, the caps $C(x_i,\eps/2)$ are disjoint and therefore the random variables $B_1, B_3 , B_5 ,...$ are independent.
    Next we exploit this independence.
    Let $k\geq k_0$, and set $K=\lfloor k/2 \rfloor $.
    Note that $K \geq 1 / 16 \eps = m^{1/(n-1)} / 16 c_6$.
    Assuming that $c_2 \leq 1/32 c_6$, we have
    \begin{align*}
        \Pr\left(\sum_{0\leq i \leq k-1} B_i \leq c_2 m^{1/(n-1)} \right) 
        &\leq \Pr\left(\sum_{1\leq i \leq K} B_{2i-1} \leq \frac{K}{2} \right)
        = \sum_{1\leq i \leq \lfloor K/2 \rfloor} \binom{K}{i}  p^i (1-p)^{K-i}.
    \end{align*}
    Now we bound $p$ by $1$, $(1-p)^{K-i}$ by $(1-p)^{K/2}$ and $\sum \binom{K}{i}$ by $2^K$, which provides us the bound
    \begin{align*}
        \Pr\left(\sum_{0\leq i \leq k-1} B_i \leq c_2 m^{-1/(n-1)} \right) 
        &\leq (2 (1-p)^{1/2})^K
        = \left(e^{\left( - \frac{1}{2} \left(\frac{c_6}{4}\right)^{n-1} + \frac{3}{2}\ln 2 \right)}\right)^K 
        \leq \left(e^{\left( -c_5^{n-1}\right)}\right)^K .
    \end{align*}
    Thus, with the bound $|X_k|\leq (17^{n-1})^k$ from lemma \ref{lem:cardXk}, and the fact that $K\geq k/2$, we get
    \begin{align*}
        \Pr\left( \diam Q(A) \leq c_2 m^{-1/(n-1)} \right)
        & \leq e^{-m/2} + \sum_{k\geq k_0} \left( e^{\left( -\frac{1}{2} (c_5)^{n-1} + (n-1)\ln 17 \right)} \right)^k 
        \\&\leq e^{-m/2} + \sum_{k\geq k_0} (e^{-(c_4)^{n-1}} )^k 
        \\&= e^{-m/2} + \frac{ e^{- k_0 c_4^{n-1}} }{ 1 - e^{-(c_4)^{n-1}} }
        \\&\leq e^{-m/2} + \frac{ e^{- \frac{m^{1/(n-1)}}{8 c_6} c_4^{n-1}} }{ 1 - e^{-c_4^{n-1}} }
        \\& \leq e^{- c_3^{n-1} m^{1/(n-1)}} . \qedhere
    \end{align*}
\end{proof}

% \section{Conclusions and Open Problems}
% \label{sec:conclusion}
% 
% We have proven upper and lower bounds
% on the combinatorial diameters of $P(A)$ and $Q(A)$.
% However, many open questions remain.
% A first open question is whether the upper bound for $P(A)$
% can be improved to remove the exponential dependence on $n$ when $m \leq 2^{O(n)}$.
% 
% Second, can the combinatorial diameter be upper bounded by $\poly(n)$ when $m$
% grows sub-exponentially in $n$?
% The analogous question for $Q(A)$, if $A$ is sampled from $\ball^n$ instead of
% $\sfe$, was answered by \cite{barany1988shape}: If $m \leq 1.125^n$ then with
% probability $1-o(1)$ any two $a$'s form an edge of $Q(A)$, and if $m \geq
% 1.4^n$ this happens with probability $o(1)$.
% 
% Last, we ask if similar bounds can be obtained for different distributions,
% e.g., if the vectors in $A$
% follow a standard Gaussian distribution.

\paragraph{Acknowledgment}
This work was done in part while the authors were participating in the following programs:
\begin{itemize}
	\item the \textit{Probability, Geometry and Computation in High Dimensions} semester at the Simons Institute for the Theory of Computing,
	\item the \textit{Interplay between High-Dimensional Geometry and Probability} trimester at the Hausdorff Institute for Mathematics,
	\item and the \textit{Discrete Optimization} trimester at the Hausdorff Institute for Mathematics.
\end{itemize} 

\bibliography{bib}

\begin{thebibliography}{10}

\bibitem{balinski1984hirsch}
Michel~L Balinski.
\newblock {The Hirsch conjecture for dual transportation polyhedra}.
\newblock {\em Mathematics of Operations Research}, 9(4):629--633, 1984.

\bibitem{barany1988shape}
Imre B{\'a}r{\'a}ny and Zolt{\'a}n F{\"u}redi.
\newblock On the shape of the convex hull of random points.
\newblock {\em Probability Theory and Related Fields}, 77(2):231--240, February
  1988.
\newblock \href {https://doi.org/10.1007/bf00334039}
  {\path{doi:10.1007/bf00334039}}.

\bibitem{barnette1}
David Barnette.
\newblock Wv paths on 3-polytopes.
\newblock {\em Journal of Combinatorial Theory}, 7(1):62--70, July 1969.
\newblock \href {https://doi.org/10.1016/s0021-9800(69)80007-4}
  {\path{doi:10.1016/s0021-9800(69)80007-4}}.

\bibitem{barnette2}
David Barnette.
\newblock An upper bound for the diameter of a polytope.
\newblock {\em Discrete Mathematics}, 10(1):9--13, 1974.
\newblock \href {https://doi.org/10.1016/0012-365x(74)90016-8}
  {\path{doi:10.1016/0012-365x(74)90016-8}}.

\bibitem{bonifas2014sub}
Nicolas Bonifas, Marco Di~Summa, Friedrich Eisenbrand, Nicolai H{\"a}hnle, and
  Martin Niemeier.
\newblock On sub-determinants and the diameter of polyhedra.
\newblock {\em Discrete \& Computational Geometry}, 52(1):102--115, 2014.

\bibitem{Borgwardt87}
Karl~Heinz Borgwardt.
\newblock {\em The simplex method: a probabilistic analysis}, volume~1 of {\em
  Algorithms and Combinatorics: Study and Research Texts}.
\newblock Springer-Verlag, Berlin, 1987.
\newblock \href {https://doi.org/10.1007/978-3-642-61578-8}
  {\path{doi:10.1007/978-3-642-61578-8}}.

\bibitem{Borgwardt99}
Karl~Heinz Borgwardt.
\newblock Erratum: A sharp upper bound for the expected number of shadow
  vertices in lp-polyhedra under orthogonal projection on two-dimensional
  planes.
\newblock {\em Mathematics of Operations Research}, 24(4):925--984, 1999.
\newblock URL: \url{http://www.jstor.org/stable/3690611}.

\bibitem{borgwardt_huhn_1999}
Karl~Heinz Borgwardt and Petra Huhn.
\newblock A lower bound on the average number of pivot-steps for solving linear
  programs valid for all variants of the simplex-algorithm.
\newblock {\em Mathematical Methods of Operations Research}, 49(2):175--210,
  April 1999.
\newblock \href {https://doi.org/10.1007/s186-1999-8373-5}
  {\path{doi:10.1007/s186-1999-8373-5}}.

\bibitem{borgwardt2018diameters}
Steffen Borgwardt, Jes{\'u}s~A De~Loera, and Elisabeth Finhold.
\newblock {The diameters of network-flow polytopes satisfy the Hirsch
  conjecture}.
\newblock {\em Mathematical Programming}, 171(1):283--309, 2018.

\bibitem{BGKKLS01}
Andreas Brieden, Peter Gritzmann, Ravindran Kannan, Victor Klee, L{\'a}szl{\'o}
  Lov{\'a}sz, and Mikl{\'o}s Simonovits.
\newblock Deterministic and randomized polynomial-time approximation of radii.
\newblock {\em Mathematika}, 48(1-2):63--105, 2001.

\bibitem{brightwell2006linear}
Graham Brightwell, Jan Van~den Heuvel, and Leen Stougie.
\newblock A linear bound on the diameter of the transportation polytope.
\newblock {\em Combinatorica}, 26(2):133--139, 2006.

\bibitem{CL19}
Cl{\'e}ment Canonne.
\newblock A short note on poisson tail-bounds.
\newblock Github repository link:
  \url{https://github.com/ccanonne/probabilitydistributiontoolbox/blob/master/poissonconcentration.pdf},
  2019.

\bibitem{dadush-hahnle}
Daniel Dadush and Nicolai H{\"a}hnle.
\newblock On the shadow simplex method for curved polyhedra.
\newblock {\em Discrete Computational Geometry}, 56(4):882--909, June 2016.
\newblock \href {https://doi.org/10.1007/s00454-016-9793-3}
  {\path{doi:10.1007/s00454-016-9793-3}}.

\bibitem{DH19}
Daniel Dadush and Sophie Huiberts.
\newblock A friendly smoothed analysis of the simplex method.
\newblock {\em SIAM Journal on Computing}, 49(5):STOC18--449, 2019.

\bibitem{del2016diameter}
Alberto Del~Pia and Carla Michini.
\newblock On the diameter of lattice polytopes.
\newblock {\em Discrete \& Computational Geometry}, 55(3):681--687, 2016.

\bibitem{deza2018improved}
Antoine Deza and Lionel Pournin.
\newblock Improved bounds on the diameter of lattice polytopes.
\newblock {\em Acta Mathematica Hungarica}, 154(2):457--469, 2018.

\bibitem{dyer1994random}
Martin Dyer and Alan Frieze.
\newblock Random walks, totally unimodular matrices, and a randomised dual
  simplex algorithm.
\newblock {\em Mathematical Programming}, 64(1):1--16, 1994.

\bibitem{EHRR10}
Friedrich Eisenbrand, Nicolai H{\"a}hnle, Alexander Razborov, and Thomas
  Rothvoss.
\newblock Diameter of polyhedra: Limits of abstraction.
\newblock {\em Mathematics of Operations Research}, 35(4):786--794, 2010.

\bibitem{glisse2016silhouette}
Marc Glisse, Sylvain Lazard, Julien Michel, and Marc Pouget.
\newblock Silhouette of a random polytope.
\newblock {\em Journal of Computational Geometry}, 7(1):14, 2016.

\bibitem{grinold1971hirsch}
Richard~C Grinold.
\newblock {The Hirsch conjecture in Leontief substitution systems}.
\newblock {\em SIAM Journal on Applied Mathematics}, 21(3):483--485, 1971.

\bibitem{janson}
Svante Janson.
\newblock Large deviations for sums of partly dependent random variables.
\newblock {\em Random Structures \& Algorithms}, 24(3):234--248, 2004.
\newblock \href {https://doi.org/https://doi.org/10.1002/rsa.20008}
  {\path{doi:https://doi.org/10.1002/rsa.20008}}.

\bibitem{kalaikleitman}
Gil Kalai and Daniel~J. Kleitman.
\newblock A quasi-polynomial bound for the diameter of graphs of polyhedra.
\newblock {\em Bull. Amer. Math. Soc.}, 26(2):315--317, July 1992.
\newblock \href {https://doi.org/10.1090/s0273-0979-1992-00285-9}
  {\path{doi:10.1090/s0273-0979-1992-00285-9}}.

\bibitem{klee1967d}
Victor Klee, David~W Walkup, et~al.
\newblock The $d$-step conjecture for polyhedra of dimension $d<6$.
\newblock {\em Acta Mathematica}, 117:53--78, 1967.

\bibitem{kleinschmidt1992diameter}
Peter Kleinschmidt and Shmuel Onn.
\newblock On the diameter of convex polytopes.
\newblock {\em Discrete mathematics}, 102(1):75--77, 1992.

\bibitem{labbe2017hirsch}
Jean-Philippe Labb{\'e}, Thibault Manneville, and Francisco Santos.
\newblock Hirsch polytopes with exponentially long combinatorial segments.
\newblock {\em Mathematical Programming}, 165(2):663--688, 2017.

\bibitem{larman}
D.G. Larman.
\newblock Paths on polytopes.
\newblock {\em Proc. London Math. Soc. (3)}, s3-20(1):161--178, January 1970.
\newblock \href {https://doi.org/10.1112/plms/s3-20.1.161}
  {\path{doi:10.1112/plms/s3-20.1.161}}.

\bibitem{michini2014hirsch}
Carla Michini and Antonio Sassano.
\newblock {The Hirsch Conjecture for the fractional stable set polytope}.
\newblock {\em Mathematical Programming}, 147(1):309--330, 2014.

\bibitem{naddef1989hirsch}
Denis Naddef.
\newblock {The Hirsch conjecture is true for (0, 1)-polytopes}.
\newblock {\em Mathematical Programming: Series A and B}, 45(1-3):109--110,
  1989.

\bibitem{narayanan2021spectral}
Hariharan Narayanan, Rikhav Shah, and Nikhil Srivastava.
\newblock A spectral approach to polytope diameter, 2021.
\newblock \href {http://arxiv.org/abs/2101.12198} {\path{arXiv:2101.12198}}.

\bibitem{RS16}
A~Reznikov and EB~Saff.
\newblock The covering radius of randomly distributed points on a manifold.
\newblock {\em International Mathematics Research Notices},
  2016(19):6065--6094, 2016.

\bibitem{sanita2018diameter}
Laura Sanit{\`a}.
\newblock The diameter of the fractional matching polytope and its hardness
  implications.
\newblock In {\em 2018 IEEE 59th Annual Symposium on Foundations of Computer
  Science (FOCS)}, pages 910--921. IEEE, 2018.

\bibitem{santos}
Francisco Santos.
\newblock {A counterexample to the Hirsch Conjecture}.
\newblock {\em Annals of Mathematics}, 176(1):383--412, July 2012.
\newblock \href {https://doi.org/10.4007/annals.2012.176.1.7}
  {\path{doi:10.4007/annals.2012.176.1.7}}.

\bibitem{schneider2008stochastic}
Rolf Schneider and Wolfgang Weil.
\newblock {\em Stochastic and integral geometry}.
\newblock Probability and its Applications (New York). Springer-Verlag, Berlin,
  2008.
\newblock \href {https://doi.org/10.1007/978-3-540-78859-1}
  {\path{doi:10.1007/978-3-540-78859-1}}.

\bibitem{jour/jacm/ST04}
Daniel~A Spielman and Shang-Hua Teng.
\newblock Smoothed analysis of algorithms: Why the simplex algorithm usually
  takes polynomial time.
\newblock {\em Journal of the ACM (JACM)}, 51(3):385--463, 2004.

\bibitem{sukegawa2019asymptotically}
Noriyoshi Sukegawa.
\newblock An asymptotically improved upper bound on the diameter of polyhedra.
\newblock {\em Discrete \& Computational Geometry}, 62(3):690--699, 2019.

\bibitem{todd2014improved}
Michael~J Todd.
\newblock {An improved Kalai--Kleitman bound for the diameter of a polyhedron}.
\newblock {\em SIAM Journal on Discrete Mathematics}, 28(4):1944--1947, 2014.

\bibitem{jour/siamjc/Vershynin09}
Roman Vershynin.
\newblock Beyond {H}irsch conjecture: walks on random polytopes and smoothed
  complexity of the simplex method.
\newblock {\em SIAM J. Comput.}, 39(2):646--678, 2009.
\newblock Preliminary version in FOCS `06.
\newblock URL: \url{http://dx.doi.org/10.1137/070683386}, \href
  {https://doi.org/10.1137/070683386} {\path{doi:10.1137/070683386}}.

\end{thebibliography}

\end{document}